\documentclass[11pt,a4paper,twoside,reqno]{amsart}

\usepackage[all]{xy}
\usepackage{graphicx}
\usepackage{epsfig, 
pstricks} 
\usepackage{float}
\usepackage{amsfonts}
\usepackage{amssymb}

\usepackage{mathrsfs}

\usepackage{fontenc}
\usepackage{amsmath}
\usepackage{latexsym}
\usepackage{enumerate}

\newcommand{\G}{{\mathcal G}}
\newcommand{\PP}{{\mathcal P}}
\newcommand{\N}{{\mathbb N}}

\newcommand{\Z}{{\mathbb Z}}
\newcommand{\R}{{\mathbb R}}

\newcommand{\RP}{{\mathbb RP}}

\newcommand{\FF}{\mathcal{F}}

\newcommand{\uu}{\mathfrak{u}}
\newcommand{\vv}{\mathfrak{v}}
\newcommand{\nnabla}{{\scriptscriptstyle \nabla}}
\newcommand{\nnablao}{{\scriptscriptstyle \nabla_0}}

\newcommand{\inj}{{\rm inj}}

\newcommand{\vol}{{\rm vol}}

\newcommand{\id}{{\rm id}}
\newcommand{\length}{{\rm length}}

\newcommand{\Diff}{{\rm Diff}}

\newcommand{\cf}{{\it cf.}}
\newcommand{\ie}{{\it i.e.}}

\numberwithin{equation}{section}

\newtheorem{theorem}{Theorem}[section]
\newtheorem{proposition}[theorem]{Proposition}
\newtheorem{corollary}[theorem]{Corollary}
\newtheorem{lemma}[theorem]{Lemma}

\theoremstyle{definition}
\newtheorem{definition}[theorem]{Definition}

\newtheorem{remark}[theorem]{Remark}

\long\def\forget#1\forgotten{} %

\title[Retraction of the space of Zoll Finsler projective planes]{Strong deformation retraction of the space of Zoll Finsler projective planes}

\author[S.~Sabourau]{St\'ephane Sabourau}

\address{Universit\'e Paris-Est,
Laboratoire d'Analyse et Math\'ematiques Appliqu\'ees (UMR 8050), 
UPEC, UPEMLV, CNRS, F-94010, Cr\'eteil, France}

\email{stephane.sabourau@u-pec.fr}



\thanks{Partially supported by the ANR grant Finsler} 

\begin{document}

\begin{abstract}
We show that the infinite-dimensional space of Zoll Finsler metrics on the projective plane strongly deformation retracts to the canonical round metric.
In particular, this space of Zoll Finsler metrics is connected.
Moreover, the strong deformation retraction arises from a deformation of the geodesic flow of every Zoll Finsler projective plane to the geodesic flow of the round metric through a family of smooth free circle actions induced by the curvature flow of the canonical round projective plane.
This construction provides a description of the geodesics of the Zoll Finsler metrics along the retraction.
\end{abstract}

\maketitle


\section{Introduction}

A \emph{Zoll metric} on a closed manifold~$M$ is a Riemannian or (reversible, quadratically convex) Finsler metric all of whose geodesics are simple closed curves of the same length.
We refer to the classical reference~\cite{besse} for an introduction to the subject and historical comments, see also~\cite{berger}.
Zoll manifolds have finite fundamental groups.
Thus, in the two-dimensional case, they are diffeomorphic to either the sphere or the projective plane.
Note that the orientable double cover of a Zoll Finsler projective plane is a Zoll Finsler two-sphere, \cf~Proposition~\ref{prop:equiv} or Remark~\ref{rem:lift}.

\medskip

The canonical round metric on the two-sphere or the projective plane is a Zoll Riemannian metric.
However, there exist Zoll Riemannian two-spheres which are not round; some are rotationally symmetric, \cf~\cite[\S4]{besse}, while others have no symmetry at all, \cf~\cite[Corollary~4.71]{besse}.
Actually, Zoll Riemannian metrics on the two-sphere (modulo isometries and rescaling) form an infinite-dimensional space.
Contrariwise, a Riemannian metric on the projective plane is a Zoll metric if and only if it has constant curvature, which follows from Green's theorem, \cf~\cite[Theorem~5.59]{besse}, since the orientable double cover of a Zoll projective plane is a Blaschke sphere.
This also holds true in higher dimension, \cf~\cite[Appendix~D]{besse}.
However, this rigidity result fails in the Finsler case.
Indeed, Zoll Finsler metrics on the projective plane (modulo isometries and rescaling) form an infinite-dimensional space, \cf~Appendix.

\medskip

The goal of this article is to study the space of Zoll Finsler metrics on the projective plane and the dynamics of their geodesic flow.
More specifically, one can ask the following question about the topology of such space: 

\medskip

\noindent \emph{Is the space of all Zoll Finsler metrics on any closed manifold connected (when nonempty)?} 

\medskip

In the Riemannian case, this is a famous question whose answer is only known for the projective plane: the canonical round metric is the only Zoll Riemannian metric on the projective plane modulo isometries and rescaling.
Even on the two-sphere, the question is wide open, \cf~\cite[Question~200]{berger}.
(Observe that an approach through the Ricci flow does not work as shown in~\cite{jane}.)
Now, in the Finsler case, Zoll metrics are much more flexible (as forementionned, their moduli space -- if nonempty -- is always infinite-dimensional, \cf~Appendix) and the question still makes sense.

\medskip

Our main result shows that the topology of the space of Zoll Finsler metrics on the projective plane is homotopically trivial.
This provides the first (positive) answer to the question above for Zoll Finsler metrics on the projective plane.

\begin{theorem} \label{theo:retractrp2}
The space of Zoll Finsler metrics on the projective plane whose geodesic length is equal to~$\pi$ strongly deformation retracts to the canonical round metric on the projective plane.

Furthermore, this strong deformation retraction is induced by the curvature flow on the canonical round projective plane.
\end{theorem}

We emphasize that the deformation retraction is not given by some abstract existence theorem but proceeds from a natural geometric construction relying on the curvature flow of the canonical round metric, \cf~Theorem~\ref{theo:flowrp2}.
In particular, this implies that the geodesics of the Zoll Finsler metrics~$F_\tau$ along the deformation retraction from a given Zoll Finsler metric~$F$ are obtained by applying the curvature flow of the canonical round metric to the geodesics of the given Zoll Finsler metric~$F$.
The construction of the deformation retraction in Theorem~\ref{theo:retractrp2} is fairly concrete.
It relies on Theorem~\ref{theo:flowrp2} below (and the material developed in the first part of the article) and a construction of Finsler metrics through the Crofton formula due to \'Alvarez Paiva and Berck~\cite{AB}.
This approach would carry over to the case of Zoll Finsler metric on the two-sphere if one could deform the geodesics of these metrics to the equators of the canonical round sphere while preserving their intersection pattern.

\medskip

The following result is a straightforward consequence of Theorem~\ref{theo:retractrp2}.
It immediately follows from a construction of~\cite{wein75} (see also~\cite[Appendix~B]{guillemin} for a more explicit statement) relying on Moser's trick.

\begin{corollary} \label{coro:conj}
Let $(F_\tau)$ be the family of Zoll Finsler metrics on the projective plane~$M$ given by applying the retraction constructed in Theorem~\ref{theo:retractrp2} to a Zoll Finsler metric~$F$ with geodesic length~$\pi$.
There exists a natural one-parameter family of (homogeneous) symplectomorphisms 
\[
\phi_\tau: T^*M \setminus \{0\} \to T^*M \setminus \{0\}
\]
with $F_\tau^* \circ \phi_\tau = g_0^*$.
In particular, the cogeodesic flows of~$F$ and~$g_0$ are symplectically conjugate.

Furthermore, the symplectomorphism~$\phi_\tau$ takes the cogeodesics of~$g_0$ to the curves obtained from the (co)-geodesics of~$F$ by applying the curvature flow of the canonical round metric.
\end{corollary}

The symplectic conjugacy of the cogeodesic flows of Zoll Finsler two-spheres (and Zoll Finsler projective planes by taking their quotient) has recently been established in~\cite{ABHS} by other means.
Therefore, the statement about the symplectic conjugacy in Corollary~\ref{coro:conj} is not new, but our approach provides extra information on the symplectomorphisms~$\phi_\tau$ and yields an alternative proof.

\medskip

In the proof of our main theorem, we will need the following theorem connecting the geodesic flows of Zoll projective planes~$M$.
In this result, we identify the unit tangent bundle~$SM$ of any Finsler metric on~$M$ with the unit tangent bundle~$S_0M$ of the canonical round metric~$g_0$ by radial projection.

\begin{theorem} \label{theo:flowrp2}
Let $(M,F)$ be a Zoll Finsler projective plane.
There exists a natural one-parameter family of smooth free $S^1$-actions~$(\rho_\tau)_{0 \leq \tau \leq 1}$ on~$S_0 M$
\[
\rho_\tau:S^1 \times S_0 M \to S_0 M
\]
between the geodesic flows of~$F$ and~$g_0$ such that every $\rho_\tau$-orbit projects to an embedding of~$S^1$ into~$M$ under the canonical projection $S_0 M \to M$.

Furthermore, this one-parameter family of actions is induced by the curvature flow on the canonical round projective plane.
\end{theorem}


Here again, we emphasize that the family of circle actions~$(\rho_\tau)$ connecting the two geodesic flows proceeds in a natural way from the curvature flow: the $\rho_\tau$-orbits correspond to the curves obtained from the $F$-geodesics by applying the curvature flow of the canonical round metric.
This construction makes the family~$(\rho_\tau)$ more trackable.

\medskip

Actually, Theorem~\ref{theo:flowrp2} directly follows from~\cite{hsu} once the intersection pattern of closed geodesics on Zoll Finsler surfaces is established, \cf~Section~\ref{sec:intersection}.
More precisely, the construction of the family of actions~$(\rho_t)$ follows from Theorem~\ref{theo:flowS2}, which is a particular case of a result of~\cite{hsu} on the curvature flow.
Still, we decided to present a proof of Theorem~\ref{theo:flowS2}, since the estimates required in our case are weaker than those established in~\cite{hsu}.

\medskip

Specifically, the construction of the family of actions~$(\rho_\tau)$ proceeds as follows.
First, we examine the infinitesimal and non-infinitesimal intersection properties of the closed geodesics of Zoll Finsler two-spheres, \cf~Section~\ref{sec:intersection}.
Then, we apply the curvature flow of the canonical round sphere to simultaneously deform these simple closed curves into equators of the round sphere.
Here, we need to assume that the geodesics of the Zoll Finsler sphere divide the sphere into two domains of the same $g_0$-area, otherwise the curves shrink to points under the curvature flow of the canonical round sphere, \cf~Theorem~\ref{theo:prop}.
(This is the case on the orientable double cover of a Zoll Finsler projective plane.)
Note also that for arbitrary metrics on the two-sphere, the curvature flow may not converge as it may oscillate between closed geodesics.
However, it does converge to equators of the round sphere when applied to simple curves dividing the round sphere into domains of the same area, see Theorem~\ref{theo:prop} for a discussion about the convergence of the curvature flow.
Lifting the curve deformations given by the curvature flow to the unit tangent bundle, we connect the geodesic flow of the Zoll Finsler metric to the geodesic flow of the canonical round metric~$g_0$.

\medskip

We do not know whether our results extend to non-reversible Finsler metrics as several arguments only work in the reversible case.
It would be interesting to clarify this point.

\medskip

\noindent \emph{Acknowledgment:} We are indebted to Juan-Carlos \'Alvarez Paiva for bringing the preprint~\cite{hsu} to our attention after we sent him a first version of our paper at the end of the summer 2015.

\section{Preliminaries}

In this preliminary section, we go over constructions related to the geodesic flow of a Zoll Finsler metric and review the main features of the curvature flow on the canonical round two-sphere.

\begin{definition} \label{def:finsler}
A (reversible) \emph{Finsler metric} on a closed manifold~$M$ is a continuous function $F:TM \to [0,\infty)$ on the tangent bundle~$TM$ of~$M$ satisfying the following properties (here, $F_x:=F_{|T_x M}$ for short):
\begin{enumerate}
\item Smoothness: $F$ is smooth outside the zero section;
\item Homogeneity: $F_x(tv) = |t| \, F_x(v)$ for every $v \in T_x M$ and $t \in \R$;
\item Quadratic convexity: for every $x \in M$, the function~$F_x^2$ has positive definite second derivatives on $T_x M \setminus \{ 0\}$, that is, for every $p,u,v \in T_x M$, the symmetric bilinear for
\[
g_p(u,v) = \frac{1}{2} \, \frac{\partial^2}{\partial s \partial t} F_x^2(p+tu+sv)_{|t=s=0}
\]
is an inner product.
\end{enumerate}
The metric~$F$ induces a Minkowski norm~$F_x$ on each tangent space~$T_x M$.
We will denote by $F^*:T^*M \to \R$ the function whose restriction to each cotangent space~$T_x^* M$ is given by the dual norm~$F_x^*$.

The quadratically convex condition (as opposed to a mere convex condition) allows us to define a \emph{geodesic flow} for~$F$ acting on the unit tangent bundle~$SM$ of~$M$, \cf~\cite[\S1]{besse}.
The geodesic flow of a Zoll Finsler manifold~$(M,F)$ of geodesic length~$2\pi$ is periodic and defines a smooth free action of~$S^1=\R/2\pi \Z$ on~$SM$
\[
\rho_F: S^1 \to \Diff(SM)
\]
given by
\[
\rho_F(\theta)(v) = \gamma_v'(\theta)
\]
where $\gamma_v$ is the (arclength parametrized) $F$-geodesic induced by~$v$.

Recall that the quotient manifold theorem asserts that if $G$ is a Lie group acting smoothly, freely and properly on a smooth manifold~$N$, then the quotient space $N/G$ is a topological manifold with a unique smooth structure such that the quotient map $N \to N/G$ is a smooth submersion.
This result applies to the $S^1$-action~$\rho_F$ of the geodesic flow of~$F$ on~$S M$.

Denote by 
\[
\Gamma_F=S M/\rho_F
\]
the quotient manifold and by 
\begin{equation} \label{eq:q}
q_F:S M \to \Gamma_F
\end{equation}
the quotient submersion.
The quotient manifold~$\Gamma_F$ represents the space of unparametrized oriented geodesics of the Zoll Finsler manifold~$(M,F)$.
When $M$ is a two-sphere, the space~$\Gamma_F$ is diffeomorphic to~$S^2$ as it follows from the homotopy exact sequence of the fibration~$q_F$, \cf~\cite[\S2.10]{besse}.
\end{definition}

Let us review the main features of the curvature flow on the canonical round two-sphere.

\begin{definition} \label{def:flow}
Let $\gamma:S^1 \to M$ be a smooth embedded curve on the canonical round two-sphere~$M$.
There exists a homotopy $\gamma_t:S^1 \to M$ evolving according to the equation
\[
\frac{\partial \gamma_t}{\partial t} = \kappa \, \mathfrak{n}
\]
where $\kappa$ is the curvature of~$\gamma$ in~$M$ and $\mathfrak{n}$ is its unit normal vector.

This flow, referred to as the \emph{curvature flow} on the canonical round two-sphere, is defined for a maximal time interval~$[0,T)$, where $T$ is finite if and only if $(\gamma_t)$ converges to a point when $t$ tends to~$T$, \cf~\cite{grayson}.
\end{definition}

We summarize the properties of the curvature flow on the canonical round two-sphere that we will need in this article as follows.

\begin{theorem}[\cite{angII}, \cite{gage}, \cite{grayson}] \label{theo:prop}
The curvature flow~$(\gamma_t)$ of an embedded closed curve~$\gamma$ on the canonical round two-sphere satisfies the following properties:
\begin{enumerate}
\item the length of~$\gamma_t$ decreases unless $\gamma$ is a geodesic, in which case the flow is constant; \label{c1}
\item the curves $\gamma_t$ remain embedded, \cf~\cite[Theorem~3.1]{gage} (see also~\cite[Theorem~1.3]{angII}); \label{c2}
\item two disjoint smooth simple curves~$\gamma_1$ and~$\gamma_2$ remain disjoint through the curvature flow, that is, $\gamma_{1,t}$ and~$\gamma_{2,t}$ are disjoint, unless one of them shrinks to a point, \cf~\cite[\S1]{angII}. \label{c3}
\item the curvature of~$\gamma_t$ converges to zero in the $C^\infty$-norm unless $\gamma_t$ converges to a point, \cf~\cite{grayson}; \label{c4}
\item if $\gamma$ divides the round sphere into two domains with the same area then the curves $\gamma_t$ also divide the round sphere into two domains with the same area, \cf~\cite[Proof of Theorem~5.1]{gage}, and converge to an equator as unparametrized curves. \label{c5}
\item if $\gamma$ does not divide the round sphere into two domains with the same area then the curvature flow~$(\gamma_t)$ converges to a point. \label{c6}
\end{enumerate}
\end{theorem}

\begin{proof}
The second part of the point~\eqref{c5} on the convergence of the curvature flow to an equator follows by combining the works of Gage~\cite{gage} and Grayson~\cite{grayson}.
Indeed, from the first part of the point~\eqref{c5}, the curves $\gamma_t$ divide the round sphere into two domains with the same area.
In particular, the curvature flow~$(\gamma_t)$ of~$\gamma$ does not converge to a point.
From the point~\eqref{c4}, its curvature converges to zero in the $C^\infty$-norm and, since the loops~$\gamma_t$ are simple, its length tends to~$2\pi$.
In particular, its total geodesic curvature $\int_{\gamma_t} \sqrt{\kappa^2+1} \, ds$ tends to~$2 \pi$, and so, is less than~$3 \pi$ for $t$ large enough.
This ensures that the two conditions of Theorem~5.1 in~\cite{gage} are satisfied.
Therefore, we conclude that the curvature flow~$(\gamma_t)$ converges to an equator as unparametrized curves. 

\medskip

For the point~\eqref{c6}, let $D_t$ be the domain of the round two-sphere bounded by the simple closed curve~$\gamma_t$ such that the orientation of~$D_t$ induces the same orientation as~$\gamma_t$ on its boundary.
By the Gauss-Bonnet formula, the area of~$D_t$ satisfies
\[
|D_t| = 2\pi - \int_{\gamma_t} \kappa_t \, ds.
\]
From~\cite[Lemma~1.3]{gage}, we have
\[
\frac{d}{dt} |D_t| = - \int_{\gamma_t} \kappa_t \, ds = |D_t| - 2 \pi.
\]
Therefore, $|D_t| = (|D_0|-2\pi) e^t + 2\pi$.
Since $|D_0| \neq 2\pi$, it follows that the curvature flow of~$\gamma$ is only defined on a finite time interval.
Hence, the result.
\end{proof}

\section{Geodesic intersections on Zoll Finsler two-spheres} \label{sec:intersection}

In this section, we examine some features satisfied by the geodesics of Zoll Finsler two-spheres.

\medskip

The following result is established in~\cite{LM} for Riemannian metrics but the proof carries over to Finsler metrics.

\begin{proposition}[{\cite[Proposition~2.21]{LM}}] \label{prop:equiv}
Let $(M,F)$ be a Finsler two-sphere.
The following assertions are equivalent:
\begin{enumerate}[(i)]
\item all the geodesics of~$F$ are simple closed curves; \label{ii}
\item all the geodesics of~$F$ are simple closed curves of the same length.
\end{enumerate}

In particular, the orientable double cover of a Zoll Finsler projective plane of geodesic length~$\pi$ is a Zoll Finsler two-sphere of geodesic length~$2\pi$.
\end{proposition}

\begin{remark} \label{rem:lift}
One could directly prove the last statement of Proposition~\ref{prop:equiv}.
Simply observe that a noncontractible geodesic on a Finsler projective plane cannot be approached by a contractible one of the same length.
Thus, all the simple closed geodesics on a Zoll Finsler projective plane lift to simple closed geodesics of twice their length.
\end{remark}

The following result clarifies the intersection pattern of geodesics on Zoll Finlser surfaces.
Although the result is not surprising, we were unable to find a reference for it in the literature.

\begin{theorem} \label{theo:2}
Let $(M,F)$ be a Zoll Finsler two-sphere.
Every pair of distinct (closed) geodesics has exactly two (transverse) intersection points.
\end{theorem}

\begin{proof}
Let $\G=\Gamma_F/\pm$ be the space of unparametrized geodesics of~$F$.
First, observe that two distinct (unparametrized) closed geodesics are either disjoint or have only transverse intersection points.
Now, fix a closed geodesic~$\gamma$ of~$F$.
For every nonnegative integer~$k$, consider the space~$\G_{\gamma,k}$ of closed geodesics different from~$\gamma$, intersecting~$\gamma$ at exactly $k$ (transverse) points.
The space~$\G_{\gamma,k}$ is clearly open in~$\G \setminus \{ \gamma \}$ as every closed geodesic close enough to a geodesic~$\gamma$ in~$\G_{\gamma,k}$ still has exactly $k$ transverse intersection points with~$\gamma$.
Observe that the open sets~$\G_{\gamma,k}$ are disjoint and cover~$\G \setminus \{ \gamma \}$ for $k$ running over all the nonnegative integers.
That is, 
\[
\G \setminus \{\gamma \} = \coprod_{k \in \N} \G_{\gamma,k}.
\]
By connectedness of~$\G \setminus \{ \gamma \}$, we conclude that $\G \setminus \{ \gamma \} = \G_{\gamma,k}$ for some nonnegative integer~$k_\gamma$.
That is, every closed geodesic different from~$\gamma$ intersects~$\gamma$ at exactly~$k_\gamma$ points.

Denote $\G_\gamma = \G_{\gamma,k_\gamma}$.
Observe that the integer~$k_\gamma$ does not depend on~$\gamma$.
Indeed, for every $\gamma_1,\gamma_2 \in \G$ with $\gamma_1 \neq \gamma_2$, we have $|\gamma_1 \cap \gamma_2|=k_{\gamma_1}$ since $\gamma_2 \in \G \setminus \{ \gamma_1 \} = \G_{\gamma_1}$, and, by symmetry, $|\gamma_1 \cap \gamma_2|=k_{\gamma_2}$ since $\gamma_1 \in \G \setminus \{ \gamma_2 \} = \G_{\gamma_2}$.
Hence, $k_{\gamma_1} = k_{\gamma_2}$.

Thus, every pair of distinct closed geodesics has exactly $k$ (transverse) intersection points, where the integer $k$ only depends on the dynamics of the geodesic flow of~$F$.
This integer is at least two for topological reasons.

Let $\gamma_1$ and $\gamma_2$ be two distinct closed geodesics.
Since the closed geodesics $\gamma_1$ and~$\gamma_2$ are simple, there exists a connected component~$D$ of $M \setminus (\gamma_1 \cup \gamma_2)$ bounded by exactly two geodesic arcs (one lying in~$\gamma_1$ and the other lying in~$\gamma_2$).
This connected component forms a bigon with endpoints~$p$ and~$q$.
Changing the parametrization of~$\gamma_1$ and~$\gamma_2$ if necessary, we can assume that the tangent vectors $v_1=\gamma_1'(0)$ and~$v_2 = \gamma_2'(0)$ based at~$p$ span a sector in~$T_p M$ pointing inside~$D$, \cf~Figure~\ref{fig:digon}.

\begin{figure}[htb] 
\def\svgwidth{7cm} 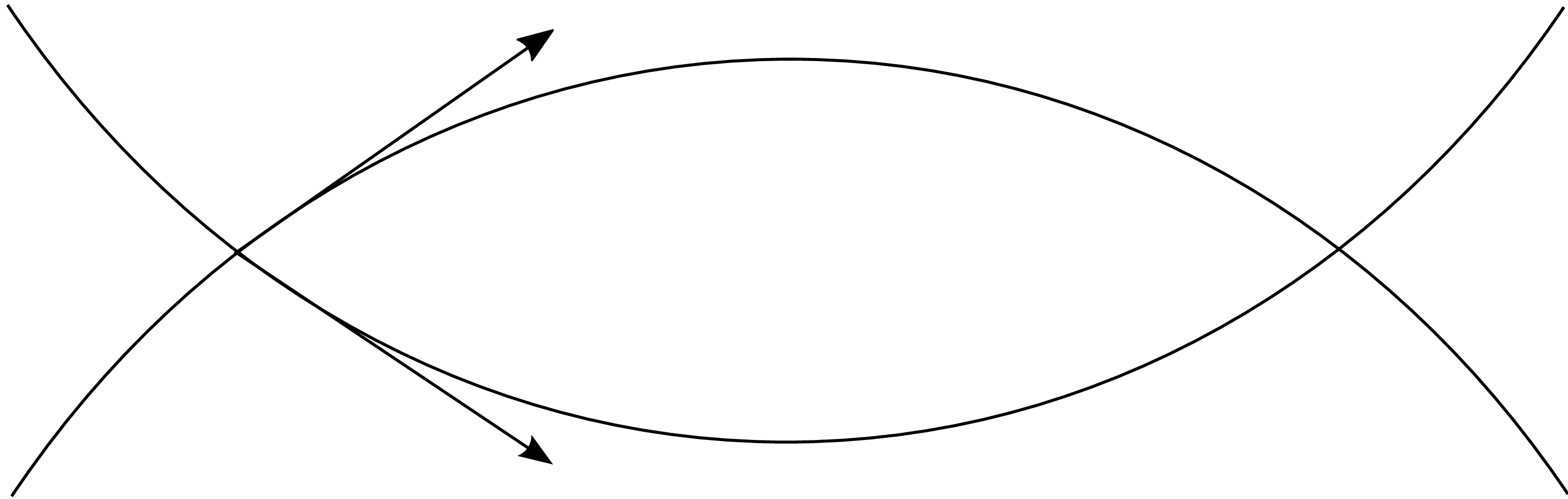
\caption{The digon~$D$} \label{fig:digon}
\end{figure}

Observe that the connected component~$D$ continuously varies with $\gamma_1$ and~$\gamma_2$ as long as $v _1$ and~$v_2$ are not collinear.
In particular, by rotating~$v_1$ to~$v_2$ and $v_2$ to~$-v_1$, we deform~$\gamma_1$ to~$\gamma_2$ and $\gamma_2$ to~$-\gamma_1$ through two homotopies of simple closed geodesics~$\gamma_1^t$ and~$\gamma_2^t$.
Through this process, the digon~$D$ bounded by the two arcs of~$\gamma_1$ and~$\gamma_2$ joining~$p$ to~$q$ and directed by~$v_1$ and~$v_2$ deforms through a family of digons~$D_t$ bounded by the two arcs of~$\gamma_1^t$ and~$\gamma_2^t$ joining~$p$ to some point~$q_t \in \gamma_1^t \cap \gamma_2^t$ and directed by~$v_1^t$ and~$v_2^t$.
The digon~$D_t$ is a connected component of $M \setminus (\gamma_1^t \cup \gamma_2^t)$.
By construction, the point~$q_t$ is the first point of intersection of~$\gamma_1^t$ and~$\gamma_2^t$ when travelling along these two geodesics from~$p$ in the directions of $v_1^t$ and~$v_2^t$.
At the final time~$t=1$, the geodesics $\gamma_1^1$ and~$\gamma_2^1$ agree with~$\gamma_2$ and~$-\gamma_1$.
Hence, $q_1$ agrees with~$q$.
This implies that the two digons~$D$ and~$D_1$, which are connected components of $M \setminus (\gamma_1 \cup \gamma_2)$, form a disk with boundary~$\gamma_1$, split by an arc of~$\gamma_2$.
As a result, the closed geodesics $\gamma_1$ and~$\gamma_2$ have exactly two intersection points, namely $p$ and~$q$.
Hence, $k=2$.
\forget

Let $\beta$ be a closed geodesic (transversaly) intersecting~$\gamma$.
Denote by~$k$ the number of intersection points between~$\gamma$ and~$\beta$.
Consider the space~$\Gamma_{\gamma,k}$ of closed geodesics $\gamma \neq \gamma$ intersecting~$\gamma$ at exactly $k$ (transverse) points.

The space~$\Gamma_{\gamma,k}$ containing~$\beta$ is nonempty.
It is clearly open in~$\Gamma \setminus \{ \gamma \}$ as every closed geodesic close enough to~$\gamma \in \Gamma_{\gamma,k}$ still has exactly $k$ transverse intersection points with~$\gamma$.
It is also closed in~$\Gamma \setminus \{ \gamma \}$.
Indeed, the intersection points of~$\gamma$ with another closed geodesic~$\gamma$ cannot be arbitrarily close.
Thus, the transverse intersection points between~$\gamma$ and~$\gamma \in \Gamma_{\gamma,k}$ cannot ``cancel out" when $\gamma$ converges to a closed geodesic~$\bar{\gamma}$ different from~$\gamma$.
Similarly, extra transverse intersection points cannot appear when $\gamma$ converges to~$\bar{\gamma}$.
Hence, the closed geodesic~$\bar{\gamma}$ lies in~$\Gamma_{\gamma,k}$.
By connectedness of~$\Gamma \setminus \{\gamma\}$, we conclude that $\Gamma_{\gamma,k} = \Gamma \setminus \{\gamma\}$.
That is, every closed geodesic different from~$\gamma$ intersects~$\gamma$ at exactly $k_\gamma$ points.

Denote $\Gamma_\gamma = \Gamma_{\gamma,k}$.
Observe that the integer~$k_\gamma$ does not depend on~$\gamma$.
Indeed, for every $\gamma_1,\gamma_2 \in \Gamma$ with $\gamma_1 \neq \gamma_2$, we have $|\gamma_1 \cap \gamma_2|=k_{\gamma_1}$ since $\gamma_2 \in \Gamma_{\gamma_1}$, and, by symmetry, $|\gamma_1 \cap \gamma_2|=k_{\gamma_2}$ since $\gamma_1 \in \Gamma_{\gamma_2}$.
Hence, $k_{\gamma_1} = k_{\gamma_2}$.

Thus, the integer~$k$ only depends on the dynamics of~$(M,\nabla)$.
It is clearly greater than or equal to two for topological reasons.

Now, consider a variation of~$\gamma$ by closed geodesics~$c_\lambda$, with $c_0=\gamma$, inducing a nontrivial normal vector field~$Y_\perp$ along~$\gamma$.
The closed geodesics~$c_\lambda$ intersect~$\gamma$ at $k$ points.
Therefore, the normal vector field~$Y_\perp$ induced by this geodesic variation of~$\gamma$ vanishes at $k$ points.
From Theorem~\ref{theo:LM}, we have~$k=2$.
\forgotten
\end{proof}


In the rest of this section, we introduce the (non-metric) notion of normal vector fields along simple loops on a surface.
We also determine the number of zeros of nontrivial normal vector fields defined by geodesic variations on a Zoll Finsler two-sphere.

\begin{definition} \label{def:Y}
Given a closed surface~$M$, let $c:S^1 \times (-\varepsilon,\varepsilon) \to M$ be a smooth map inducing a smooth variation of embedded curves $c_\lambda= c(.,\lambda)$ with $\lambda \in (-\varepsilon,\varepsilon)$.
Here, the curves~$c_\lambda$ are not necessarily geodesics.
Define the following vector field $Y \in \Gamma(c_0^* TM)$ along~$c_0$ as
\[
Y(\theta) = \frac{\partial c}{\partial \lambda}(\theta,0)
\]
for every $\theta \in S^1$.
When the curves~$c_\lambda$ are geodesics for some Finsler metric~$F$ on~$M$, the vector field~$Y$ represents the Jacobi field along~$c_0$ generated by the geodesic variation~$(c_\lambda)$, \cf~\cite[\S11.2]{shen}.
The vector field~$Y$ induces a \emph{normal vector field} $Y_\perp$ along~$c_0$ defined as 
\[
Y_\perp(\theta) \equiv Y(\theta) \mbox{ mod } \R.c_0'(\theta)
\]
for every $\theta \in S^1$, where $Y_\perp(\theta)$ lies in the quotient of the plane~$T_{c_0(\theta)} M$ by the vector line $\R .c_0'(\theta)$ generated by~$c_0'(\theta)$.
On an orientable surface, a normal vector field along~$c_0$ is merely a function.
\end{definition}


We start with the following observation showing that the notion of normal vector field extends to variations of unparametrized (oriented or unoriented) embedded curves.

\begin{lemma}
Let $c_0$ be a curve in a closed manifold~$M$.
The normal vector field~$Y_\perp$ along~$c_0$ induced by a curve variation~$(c_\lambda)$ does not depend on the parametrization of the curves~$c_\lambda$.
\end{lemma}

\begin{proof}
Consider a variation of curves $\bar{c}_\lambda(\cdot)=c_\lambda(\theta(\cdot,\lambda))$, where $\theta(\cdot,\lambda)$ represents a regular change of parameter.
At $\theta = \theta(\bar{\theta},0)$, the points $\bar{c}_0(\bar{\theta})$ and~$c_0(\theta)$ agree as well as the lines generated by $\bar{c}'_0(\bar{\theta})$ and~$c'_0(\theta)$.
Now, the vector field~$\bar{Y}$ induced by the curve variation~$(\bar{c}_\lambda)$ satisfies
\begin{align*}
\bar{Y}(\bar{\theta}) & = \frac{\partial \bar{c}}{\partial \lambda}(\bar{\theta},0) \\
 & = \frac{\partial c}{\partial \lambda}(\theta,0) + \frac{\partial \theta}{\partial \lambda}(\bar{\theta},0) \, c_0'(\theta) \\
 & \equiv Y(\theta) \mbox{ mod } \R.c_0'(\theta)
\end{align*}
Hence, $\bar{Y}_\perp(\bar{\theta}) = Y_\perp(\theta)$ at the point $\bar{c}_0(\bar{\theta}) = c_0(\theta)$.
\end{proof}

The following property satisfied by every normal Jacobi vector field~$Y_\perp$ of a Zoll Finsler two-sphere can be seen as an infinitesimal version of Theorem~\ref{theo:2}.

\begin{theorem}[{\cite[Theorem~2.15]{LM}}] \label{theo:LM}
Let $(M,F)$ be a Zoll Finsler two-sphere.
Every nontrivial normal Jacobi vector field~$Y_\perp$ induced by a variation of an unparametrized (oriented or unoriented) geodesic~$\gamma$ has exactly two zeros.

Furthermore, the zeros of~$Y_\perp$ are simple, that is, the vector fields $Y_\perp$ and $(Y_\perp)'$ do not simultaneously vanish.
\end{theorem}

\begin{proof}
The first statement of the proposition is established in~\cite[Theorem~2.15]{LM} for Zoll torsion-free affine connexions on the two-sphere.
The arguments carry over in our setting.
For the sake of the reader and since the arguments are so elegant, we briefly reproduce them.

Without loss of generality, we can assume that $\gamma$ is an unparametrized oriented geodesic.
Consider the quotient submersion $q:SM \to \Gamma$ induced by the geodesic flow of~$F$, where $\Gamma = \Gamma_F$ represents the space of unparametrized oriented geodesics of~$F$, \cf~\eqref{eq:q}.
Denote by~$PT\Gamma$ the projectivized tangent space of~$\Gamma$.
The submersion $q:SM \to \Gamma$ factors through a map $\varphi:SM \to PT\Gamma$ under the canonical projection $PT\Gamma \to \Gamma$.
That is, the following diagram is commutative

\begin{center}
\mbox{ } \hfill
\xymatrix{
 & PT\Gamma \ar[d] \\
SM \ar[ru]^\varphi \ar[r]_q & \Gamma
}
\hfill \mbox{ }
\end{center}

The map~$\varphi$ can be defined as follows.
Identify the tangent plane of~$\Gamma$ at~$\gamma$ with the space of normal Jacobi fields along~$\gamma$, \cf~\cite[Proposition~2.13]{besse}.
With this identification, the map~$\varphi$ takes a unit vector $v \in S_x M$ with basepoint~$x \in M$ to the class of normal Jacobi vector fields along~$\gamma_v$ vanishing at~$x$.
Note that the map~$\varphi$ takes every orbit of the geodesic flow of~$F$ to a different fiber of the projection $PT \Gamma \to \Gamma$.
Observe also that the map~$\varphi$ is a local diffeomorphism and so a covering since $SM$ is compact.

The index of the covering is given by
\[
\frac{|\pi_1(PT \Gamma)|}{|\pi_1(SM)|} = \frac{4}{2} = 2
\]
since $\Gamma \simeq S^2$ and $SM \simeq \RP^3$.

Now, two vectors $u$ and~$v$ of~$SM$ are sent by~$\varphi$ to the same class~$[Y_\perp]$ of a nontrivial normal Jacobi field~$Y_\perp$ along a geodesic~$\gamma$ if and only if $\gamma_u$ and~$\gamma_v$ represent the same unparametrized oriented geodesic~$\gamma$, and their basepoints are zeros of~$Y_\perp$.

By definition of the index of a covering, every class of a nontrivial normal Jacobi field along~$\gamma$ has two preimages by~$\varphi$.
It follows that $Y_\perp$ has exactly two zeros.

\medskip

For the second statement, recall that every Jacobi field~$Y$ along~$c_0$ satisfies a second-order linear differential equation, and so does the normal vector field~$Y_\perp$, \cf~\cite[Equation~(3)]{LM} or~\cite[\S11.2]{shen}, namely
\[
Y''_\perp + \kappa \, Y_\perp = 0
\]
where $\kappa$ is a smooth function.
Hence, the zeros of~$Y_\perp$ are simple unless $Y_\perp$ is trivial.
\end{proof}

\section{Curvature flow and circle action deformations on the unit tangent bundles of Zoll Finsler two-spheres} \label{sec:flow}

\forget
We define a natural circle bundle over~$M$ on which $S^1$ acts under the geodesic flow of~$\nabla$ as follows.

\begin{definition}
For every $v \in TM$, denote by $\gamma_v:\R \to M$ the affine geodesic induced by~$v$, that is, with $\gamma'_v(0)=v$.
By definition of a Zoll manifold, the affine geodesic~$\gamma_v$ is periodic with minimal period~$T=T(v)$.
From~\cite[Proposition~2.6]{LM}, the minimal period~$T$ is smooth with respect to~$v \in TM$.
Since $\gamma_{\lambda v}(s) = \gamma_v(\lambda s)$ for every $s, \lambda \in \R$, the affine geodesic~$\gamma_{\lambda v}$ is periodic with minimal period~$\frac{T}{\lambda}$.
Therefore, by the submersion theorem, there exists a sub-bundle $S M \subset TM$ over~$M$, namely a circle bundle, such that every affine geodesic~$\gamma_v$ with $v \in S M$ is periodic with minimal period~$2\pi$.
That is, the geodesic flow of~$\nabla$ induces a smooth free action of~$S^1=\R/2\pi \Z$ on the circle bundle~$S M$ defined as
\begin{equation} \label{eq:action1}
\begin{array}{ccc}
S^1 \times S M & \to & S M \\
(s,v) & \mapsto & \gamma_v'(s)
\end{array}
\end{equation}
for every $v \in S M$ and $s \in S^1=\R/2\pi \Z$.

\medskip

We do not claim that the restriction of~$S M$ to any tangent plane of~$M$ bounds a convex set.
However, when the Zoll connection~$\nabla$ arises from a Zoll Finsler metric~$F$ of geodesic length~$2\pi$, the circle bundle~$S M$ agrees with the unit tangent bundle~$SM$ of~$(M,F)$.

\medskip

Let $(M,F)$ be a Zoll Finsler two-sphere.
Since every geodesic on~$(M,F)$ is periodic with the same minimal period~$2\pi$, the geodesic flow of~$F$ induces a smooth free action of~$S^1=\R/2\pi \Z$ on the unit tangent bundle~$U_F M$ of~$M$ with respect to the Finsler metric~$F$:
\[
\varphi_F : U_F M \times S^1 \to U_F M
\]
with
\begin{equation} \label{eq:phiF}
\varphi_F(v,s) = \gamma'_v(s)
\end{equation}
for every $v \in U_F M$ and $s \in S^1=\R/2\pi \Z$, where $\gamma_v$ is the arclength parametrized $F$-geodesic induced by~$v$, that is, $\gamma'_v(0)=v$.

For every $v \in U_F M$ and $s \in S^1$, the geodesics $\gamma_v$ and~$\gamma_w$ induced by~$v$ and $w=\varphi_F(v,s)$ represent the same oriented (unparametrized) simple closed geodesic.

Let $\Gamma = U_F M / S^1$ be the space of oriented unparametrized closed geodesics on the Zoll Finsler two-sphere~$(M,F)$.
The space~$\Gamma$ is diffeomorphic to the projective plane~$S^2$.
\forgotten

By analyzing the parabolic partial differential equation satisfied by the curvature flow of the canonical round two-sphere, we show that this flow induces an isotopy of diffeomorphisms of the unit tangent bundles of balanced Zoll Finsler spheres. 
As mentioned in the introduction, several results of this section can be derived from~\cite{hsu}.

\medskip

Let $(M,F)$ be a Zoll Finsler two-sphere.
The unit tangent bundle~$SM$ of~$(M,F)$ naturally identifies with the unit tangent bundle $S_0 M$ of the canonical two-sphere~$(M,g_0)$ by radial projection on each tangent plane.
More precisely, the  two radial projections 
\begin{equation} \label{eq:flat}
\pi_\flat:S M \to S_0 M
\end{equation}
and 
\begin{equation} \label{eq:sharp}
\pi_\sharp:S_0 M \to S M
\end{equation}
are reciprocal to each other.
We will refer to these diffeomorphisms as \emph{musical diffeomorphisms}.
With this identification, the smooth free action of~$S^1$ on~$S M$ given by the geodesic flow~$F$ induces a smooth free action on~$S_0 M$ by conjugacy by the musical diffeomorphisms $\pi_\flat$ and~$\pi_\sharp$.
This $S^1$-action is denoted by
\[
\rho:S^1 \to \Diff(S_0 M)
\]
and defined as
\begin{equation} \label{eq:rho}
\rho(\theta)(v) = \pi_\flat[\gamma_{\pi_\sharp(v)}'(\theta)]
\end{equation}
for every $\theta \in S^1$ and $v \in S_0 M$.
Observe that the orbits of the actions of~$S^1$ on~$S M$ and~$S_0 M$ project down to embedded closed curves in~$M$, namely the $F$-geodesics~$\gamma_v$.

\begin{definition}
A Zoll Finsler two-sphere $(M,F)$ is \emph{balanced} if every $F$-geodesic of~$M$ divides the round sphere into two domains $D_1$ and~$D_2$ with the same $g_0$-area, where $g_0$ is the canonical round metric.
This property is satisfied if $(M,F)$ is invariant under the antipodal map, that is, if it is the orientable double cover of a Zoll Finsler projective plane.
\end{definition}

\begin{remark}
We introduce the notion of balanced Zoll Finsler two-sphere for the following reason.
On a balanced Zoll Finsler two-sphere~$(M,F)$, the simple closed geodesics~$\gamma_v$ induced by the vectors $v \in S M$ converge to equators of the canonical round sphere~$(M,g_0)$ through the curvature flow~$\gamma_v^t$ of~$(M,g_0)$, \cf~Theorem~\ref{theo:prop}.\eqref{c5}.
While if $(M,F)$ is not balanced, the convergence does not hold anymore since a simple closed curve not dividing the round sphere into two domains with the same area 
shrinks to a point through the curvature flow of the round sphere, \cf~Theorem~\ref{theo:prop}.\eqref{c6}.
\end{remark}

On every balanced Zoll Finsler two-sphere~$(M,F)$, consider the map 
\[
\Psi_t: S_0 M \to S_0 M
\]
defined as
\[
\Psi_t(v) = \pi_\flat [ (\gamma_{\pi_\sharp(v)}^t)'(0) ]
\]
for every $v \in S_0 M$ and $t \in [0,\infty)$.
Here, $\gamma_v$ is the $F$-geodesic induced by~$v$ and $(\gamma_v^t)$ is the curvature flow of~$\gamma_v$ on the canonical round sphere, \cf~Definition~\ref{def:flow}.
Note that $\Psi_0$ is the identity map on~$S_0 M$.

\forget
Similar to~\eqref{eq:PhiF0}, for every balanced Zoll two-sphere, we define
\begin{align*}
S_0 M \times S^1 \times [0,\infty) & \to S_0 M \\
(v,s,t) & \mapsto \Phi_\nnabla^t(v,s) = \pi_\flat [ (\gamma_{\pi_\sharp(v)}^t)'(s) ]
\end{align*}
Note that $\Phi_\nnabla^0=\Phi_\nnabla$.

Consider the map $\Psi: S_0 M \times [0,\infty) \to S_0 M$ defined as
\[
\Psi(v,t) = \Phi_\nnabla^t(v,0)
\]
for every $v \in S_0 M$ and $t \in [0,\infty)$.

Let $\Psi_t=\Psi(\cdot,t): S_0 M \to S_0 M$.

\begin{theorem}
Let $(M,\nabla)$ be a balanced Zoll two-sphere.
Then the maps
\[
\Psi_t: S_0 M \to S_0 M
\]
induced by the curvature flow define a one-parameter family of diffeomorphisms with $\Psi_0=\id$.

Furthermore, the canonical projection $\Pi:S_0 M \to M$ takes every orbit of the $\Phi^t_\nnabla$-action to an embedding of~$S^1$ into~$M$.
\end{theorem}

\begin{remark}
The curvature flow on the round sphere induces a one-parameter family of diffeomorphisms 
\[
\Psi_t: S_0 M \to S_0 M
\]
between $\Psi_0=\id$ and $\Psi_\infty=\Psi$ such that the geodesic flows of $\nabla$ and~$\nabla_0$ are conjugate by~$\Psi$, that is, 
\[
\Psi \circ \Phi_\nnabla(v,s) = \Phi_{\nnablao}(\Psi(v,s))
\]
for every $v \in S_0 M$ and $s \in S^1$.

Furthermore, the canonical projection $\Pi:S_0 M \to M$ takes every orbit of the $\Phi_t$-actions to an embedding of~$S^1$ into~$M$, where 
\[
\Phi_t: S_0 M \times S^1 \to S_0 M
\]
represents the free action of~$S^1$ on~$S_0 M$ given by
\[
\Phi_t(v,s) = \Psi_t^{-1} [ \Phi_{\nnablao}(\Psi_t(v),s) ]
\]
for every $v \in S_0 M$ and $s \in S^1$.
\end{remark}

\forgotten

\forget

\medskip

First, we show the following result.

\begin{proposition} \label{prop:inj}
Let $(M,F)$ be a balanced Zoll Finsler two-sphere.
For every $t \in [0,\infty)$, the map $\Psi_t: S_0 M \to S_0 M$ induced by the curvature flow of the round sphere is injective.
\end{proposition}

\begin{proof}
By definition, $\Psi_t(.) = \pi_\flat[(\gamma_{\pi_\sharp(.)}^t)'(0)]$, where the maps $\pi_\flat$ and $\pi_\sharp$ are diffeomorphisms.
Thus, the map~$\Psi_t$ is injective if and only if the map 
\[
\Xi_t:S M \to TM
\]
defined as $\Xi_t(v) = (\gamma_v^t)'(0)$ is injective.
Let $v,w \in S M$ with $v \neq w$.

Suppose that $w=\pm \gamma_v'(\tau)$ for some $\tau \in (0,2 \pi)$.
Then the geodesics~$\gamma_v$ and~$\gamma_w$ represent the same unparametrized geodesic.
More precisely, $\gamma_w(\theta) = \gamma_v(\tau \pm \theta)$ for every $\theta \in S^1$.
This relation is preserved under the curvature flow, that is, $\gamma_w^t(\theta) = \gamma_v^t(\tau \pm \theta)$ for every $t \in [0,\infty)$ and $\theta \in S^1$.
Moreover, by Theorem~\ref{theo:prop}.\eqref{c2}, an embedded curve remains embedded under the curvature flow.
Hence, $\gamma_v^t(0)$ is distinct from $\gamma_w^t(0)=\gamma_v^t(\tau)$.
In particular, $(\gamma_v^t)'(0) \neq (\gamma_w^t)'(0)$.

Suppose otherwise, that is, the vectors~$w$ and $\gamma_v'(\tau)$ are not collinear for every~$\tau$.
Then the geodesics $\gamma_v$ and~$\gamma_w$ represent different unparametrized geodesics.
By Theorem~\ref{theo:2}, the closed geodesics $\gamma_v$ and~$\gamma_w$ have exactly two intersection points.
From~\cite{angII}, the number of intersection points does not increase under the curvature flow, that is, $|\gamma_v^t \cap \gamma_w^t| \leq 2$.
Now, since the simple curves $\gamma_v^t$ and $\gamma_w^t$ divide the canonical round sphere into two domains of equal area for $t=0$ and therefore for every $t \geq 0$, \cf~Theorem~\ref{theo:prop}.\eqref{c5}, we have $|\gamma_v^t \cap \gamma_w^t| \geq 2$.
Hence, the curves $\gamma_v^t$ and~$\gamma_w^t$ have exactly two intersection points for every~$t \geq 0$.
It follows from~\cite{angII} that $\gamma_v^t$ and $\gamma_w^t$ do not have any pair of vectors pointing in the same direction (oriented tangency) for any $t \in [0,\infty)$, otherwise their number of intersection points would decrease with~$t$, which is impossible.
Thus, $(\gamma_v^t)'(0) \neq (\gamma_w^t)'(0)$.
\end{proof}

\forgotten

\medskip

We will need the following classical result about the number of zeros of a parabolic partial differential equation.

\begin{theorem}[{see \cite[Theorem~C]{ang-crelle}}] \label{theo:ang-crelle}
Let $u: S^1 \times [0,T] \to \R$ be a bounded solution of the equation
\[
u_t = a(x,t) \, u_{xx} + b(x,t) \, u_x + c(x,t) \, u
\]
where $a$, $a^{-1}$, $a_t$, $a_x$, $a_{xx}$, $b$, $b_t$, $b_x$ and $c$ are bounded functions.
Then, for every $t  \in (0,T)$, the number~$z(t)$ of zeros of $u(.,t)$ is finite.

Furthermore, if both $u$ and $u_x$ vanish at $(x_0,t_0)$ then $z(t_-) > z(t_+)$ for every $t_- < t_0 < t_+$.
That is, the number of zeros decreases whenever a multiple zero occurs.
\end{theorem}

We can now show the following result.

\begin{proposition} \label{prop:immersion}
Let $(M,F)$ be a balanced Zoll Finsler two-sphere.
For every $t \in [0,\infty)$, the map $\Psi_t: S_0 M \to S_0 M$ induced by the curvature flow of the canonical round sphere is an immersion.
\end{proposition}

\begin{proof}
By definition, $\Psi_t(.) = \pi_\flat[(\gamma_{\pi_\sharp(.)}^t)'(0)]$, where the maps $\pi_\flat$ and $\pi_\sharp$ are diffeomorphisms.
Thus, the map~$\Psi_t$ is an immersion if and only if the map 
\[
\Xi_t:S M \to TM
\]
defined as $\Xi_t(v) = (\gamma_v^t)'(0)$ is an immersion.

For $t=0$, this clearly holds true.
Indeed, by construction, $\Xi_0(v) = (\gamma_v^0)'(0)=v$ for every $v \in S M$.
That is, the map~$\Xi_0$ is the inclusion map and so is an immersion.


\medskip

Fix $v \in S M$ and $\tau \in (0,\infty)$.
Let $w=w(\lambda)$ be a smooth curve in $S M$ with $w(0)=v$.
Denote $\nu = w'(0)$.
For the sake of simplicity, we will sometimes write $\gamma_\lambda$ for~$\gamma_{w(\lambda)}$.
Note that $\gamma_0 = \gamma_v$.
We want to show that the differential $d \, \Xi_\tau(v)$ of~$\Xi_\tau$ at~$v$ is injective.
That is, if the derivative $d \, \Xi_\tau(v)(\nu)$ of~$\Xi_\tau(w(\lambda))$ vanishes at~$\lambda=0$ then the vector $\nu=w'(0)$ of~$T S M$ is zero.
The idea is to write down in local coordinates the partial differential equation satisfied by~$\Xi_\tau(w)$ and to study the evolution of the normal Jacobi field given by the geodesic variation~$(\gamma_w)$.

\medskip

By Theorem~\ref{theo:prop}.\eqref{c2}, the curve~$\gamma_v^\tau$ defines an embedding of~$S^1$ into~$M$.
This embedding extends to an embedding $h:S^1 \times (-1,1) \to M$ of a cylinder onto a collar neighborhood of~$\gamma_v^\tau$, which gives rise to a normal coordinate system with $h(.,0) = \gamma_v^\tau$ in the canonical round sphere.

In this normal coordinate system around~$\gamma_v^\tau$, every curve~$\gamma_\lambda^t$ with $(\lambda,t)$ close enough to~$(0,\tau)$ can be represented in a nonparametric way as the graph 
\[
\{ (x,\uu(x,t,\lambda)) \in S^1 \times (-1,1) \mid x \in S^1 \}
\]
of a function $\uu(.,t,\lambda)$ over~$S^1$.
Observe that $\uu(x,\tau,0)=0$ for every $x \in S^1$.
From~\cite[Eq.~(3.2)]{angI} or \cite[Appendix]{gage}, the function 
\[
\uu:S^1 \times (\tau- \delta,\tau+\delta) \times (-\varepsilon,\varepsilon) \to (-1,1)
\]
satisfies the following parabolic partial differential equation of the curvature flow:
\[
\uu_t = \FF(x,\uu,\uu_x,\uu_{xx})
\]
where $\FF$ is a smooth function defined on $S^1 \times (-1,1) \times \R^2$ with \mbox{$\FF_q(x,u,p,q) > 0$}, which can be expressed in terms of the coefficients of the canonical round metric in the normal coordinate system.
Here, the subscript notations refer to partial differentiations.

In a parametric representation, the abscisse of~$\gamma_\lambda^t$ is a function of the parameter~$\theta$, that is, $x=x(\theta,t,\lambda)$ with $x(\theta,\tau,0)=\theta$.
Thus,
\[
\gamma_\lambda^t(\theta) = (x(\theta,t,\lambda),\uu(x(\theta,t,\lambda),t,\lambda)).
\]
Differentiating this expression with respect to~$\theta$ yields the tangent vector~$(\gamma_\lambda^t)'(\theta)$ which can be represented as
\[
(\gamma_\lambda^t)'(\theta) = (x(\theta,t,\lambda),\uu(x(\theta,t,\lambda),t,\lambda),x_\theta(\theta,t,\lambda),\uu_x(x(\theta,t,\lambda),t,\lambda) \, x_\theta(\theta,t,\lambda))
\]
or
\begin{equation} \label{eq:gamma'}
(\gamma_\lambda^t)'(\theta) = (x,\uu,x_\theta,\uu_x \, x_\theta)
\end{equation}
for short.

Note that $\Xi_t(w) = (\gamma_w^t)'(0)$.
Thus, the differential of~$\Xi_t$ at~$v$ in the direction~$\nu=w'(0)$ is obtained by differentiating the relation~\eqref{eq:gamma'} with respect to~$\lambda$ at $\lambda=0$.
That is,
\[
d \, \Xi_t(v)(\nu) = (x_\lambda, \uu_x \, x_\lambda + \uu_\lambda, x_{\theta \lambda}, \uu_{xx} \, x_\theta \, x_\lambda + \uu_{x \lambda} \, x_\theta + \uu_x \, x_{\theta \lambda})
\]
evaluated at~$(0,t,0)$.
Observing that $x_\theta(0,\tau,0)=1$, we simplify this expression as follows
\[
d \, \Xi_\tau(v)(\nu) = (x_\lambda, \uu_x \, x_\lambda + \uu_\lambda, x_{\theta \lambda}, \uu_{xx} \, x_\lambda + \uu_{x \lambda} + \uu_x \, x_{\theta \lambda}).
\]

Now, suppose that $\nu$ lies in the kernel of the differential~$d \, \Xi_\tau(v)$ of~$\Xi_\tau$ at~$v$, that is, $d \, \Xi_\tau(v)(\nu) = 0$.
In this case, the functions $x_\lambda$, $\uu_\lambda$, $x_{\theta \lambda}$ and $\uu_{x\lambda}$ vanish at~$(0,\tau,0)$.
Hence, both $\vv$ and $\vv_x$ vanish at $(0,\tau,0)$, where $\vv=\uu_\lambda$.
That is, the function~$\vv$ has a multiple zero at~$(0,\tau,0)$.

\medskip

Now, in a more intrinsic way, the zeros of~$\vv$ can be related to the zeros of the normal vector field induced by the curve variation~$(\gamma_\lambda^t)$ as follows.
The vector field along~$\gamma_v^t$ induced by the curve variation $(\gamma_\lambda^t)$, \cf~Definition~\ref{def:Y}, is given by
\begin{equation} \label{eq:Yt}
Y^t(\theta) = \frac{\partial}{\partial \lambda} \gamma_\lambda^t(\theta)_{|\lambda=0} = (x,\uu,x_\lambda,\uu_x \, x_\lambda + \uu_\lambda)
\end{equation}
evaluated at~$(\theta,t,0)$.
As $x_\theta(\theta,t,0) \neq 0$ for $t$ close enough to~$\tau$, it follows from the expression of~$(\gamma_v^t)'$ and~$Y^t$, \cf~\eqref{eq:gamma'} and~\eqref{eq:Yt}, that the normal vector field~$Y_\perp^t$ defined in Definition~\ref{def:Y} vanishes if and only if $\vv=\uu_\lambda$ vanishes.
More precisely,
\begin{equation} \label{eq:equiv}
Y_\perp^t(\theta) =0 \Leftrightarrow \vv(x,t) = 0
\end{equation}
where $\vv(x,t) = \uu_\lambda(x,t,0)$ and $x = x(\theta,t,0)$.

\medskip

The number of zeros of~$Y_\perp^t$ is given by the following result.

\begin{lemma} \label{lem:2}
The normal vector field~$Y_\perp^t$ has exactly two zeros along~$\gamma_v^t$ for every~$t \geq 0$.
\end{lemma}

\begin{proof}
At $t=0$, the curves~$\gamma_\lambda^0$ are geodesic for the Zoll Finsler metric~$F$.
It follows from Theorem~\ref{theo:LM} that $Y_\perp^0$ has exactly two zeros along~$\gamma_v$.
Moreover, these zeros are simple.
By the implicit function theorem, we deduce that $Y_\perp^t$ has exactly two zeros along~$\gamma_v^t$ for every $t>0$ small enough.

Let us examine how the number of zeros of~$Y_\perp^t$ evolves with $t$ for every \mbox{$t > 0$}.
Differentiating the partial differential equation
\[
\uu_t = \FF(x,\uu,\uu_x,\uu_{xx})
\]
with respect to~$\lambda$ yields the following expression
\[
\uu_{\lambda t} = \FF_u(x,\uu,\uu_x,\uu_{xx}) \, \uu_\lambda + \FF_p(x,\uu,\uu_x,\uu_{xx}) \, \uu_{\lambda x} + \FF_q(x,\uu,\uu_x,\uu_{xx}) \, \uu_{\lambda xx}.
\]
Thus, the function $\vv=\uu_\lambda$ satisfies the parabolic partial differential equation
\begin{equation} \label{eq:v}
\vv_t = a(x,t,\lambda) \, \vv_{xx} + b(x,t,\lambda) \, \vv_x + c(x,t,\lambda) \, \vv
\end{equation}
where $a = \FF_q(x,\uu,\uu_x,\uu_{xx})$, $b = \FF_p(x,\uu,\uu_x,\uu_{xx})$ and $c = \FF_u(x,\uu,\uu_x,\uu_{xx})$.

By Theorem~\ref{theo:ang-crelle}, the number of zeros of~$\vv(.,t)$ is nonincreasing with~$t$.
Therefore, the number of zeros of the normal vector field~$Y_\perp^t$ along~$\gamma_v^t$ is nonincreasing too from the relation~\eqref{eq:equiv}.
Since $Y_\perp^t$ has exactly two zeros for $t$ small enough, it follows that $Y_\perp^t$ has at most two zeros for every $t>0$.

Now, if $Y_\perp^t$ had less than two zeros, then all the curves~$(\gamma_\lambda^t)$ would be on one side of the simple loop~$\gamma_v^t$ for every $\lambda > 0$ small enough (at least to the first order).
This is impossible since $\gamma_v^t$ and $\gamma_w^t$ divide the round sphere into two domains of the same area.
Therefore, the vector field~$Y_\perp^t$ has exactly two zeros along~$\gamma_v^t$ for every~$t$.
\end{proof}

Let us continue the proof of Proposition~\ref{prop:immersion}.
Combined with~\eqref{eq:equiv}, Lemma~\ref{lem:2} shows that the function~$\vv(.,t)$ has a constant number of zeros, namely two, for every~$t \geq 0$.
Since $\vv$ satisfies the parabolic partial differential equation~\eqref{eq:v}, we deduce from Theorem~\ref{theo:ang-crelle} that the functions $\vv$ and~$\vv_x$ do not simultaneously vanish.
Thus, the differential of~$\Xi_\tau$ at~$v$ is injective.
Hence the result.
\end{proof}

The previous propositions yield the following result.

\begin{theorem} \label{theo:diffeo}
Let $(M,F)$ be a balanced Zoll Finsler two-sphere.
For every $t \in [0,\infty)$, the map $\Psi_t: S_0 M \to S_0 M$ induced by the curvature flow of the canonical round sphere is a diffeomorphism.
\end{theorem}

\begin{proof}
\forget
The map $\Psi_t$ is an injective immersion from Proposition~\ref{prop:inj} and Proposition~\ref{prop:immersion}.
Now, since $\Psi_0$ agrees with the identity map on~$S_0 M$ and the degree is preserved under homotopies, the degree of~$\Psi_t$ is nontrivial.
Thus, the map $\Psi_t$ is surjective for every $t \geq 0$ for topological reasons.
It follows from the global inverse function theorem that $\Psi_t$ is a diffeomorphism for every~$t \geq 0 $.

We can also argue as follows without relying on Proposition~\ref{prop:inj}.
\forgotten
From Proposition~\ref{prop:immersion} and since $S_0 M$ is compact, the map~$\Psi_t$ is a proper local diffeomorphism.
Therefore, it is a covering map.
Now, the map~$\Psi_t$ is $\pi_1$-injective (this is clearly the case for~$\Psi_0$ and this property is preserved under homotopy).
Hence, the covering~$\Psi_t$ is a diffeomorphism for every~$t \geq 0$.
\end{proof}

\forget
\begin{remark}
The one-parameter family of diffeomorphisms $\Psi_t : S_0 M \to S_0 M$ may not converge to a diffeomorphism of~$S_0 M$.
For instance, the $F$-geodesics converge to equators of the round sphere only as unparametrized curves, \cf~Theorem~\ref{theo:prop}.\eqref{c5}.
\end{remark}
\forgotten

\forget
Even though the diffeomorphisms $\Psi_t : S_{\nabla_0} M \to S_{\nabla_0} M$ do not necessarily converge to a diffeomorphism of~$S_{\nabla_0} M$, the $\rho_t$-actions of~$S^1$ on~$S_{\nabla_0} M$ converge to the standard action of~$S^1$ on~$S_{\nabla_0} M$ up to reparametrization.
\forgotten

This isotopy of diffeomorphisms allows us to define a deformation~$\rho_t$ of the geodesic flow $\rho_0=\rho$ of balanced Zoll Finsler spheres, \cf~\eqref{eq:rho}, to the geodesic flow of the canonical round sphere as follows.

\medskip

Let $(M,F)$ be a balanced Zoll Finsler two-sphere.
For every $v \in S_0 M$, consider the unique curve~$\gamma_u^t$ tangent to~$v$ at~$\theta=0$ and pointing in the same direction as~$v$.
That is, $u=\Psi_t^{-1}(v)$ under the identification $S_0M=SM$.
Reparametrize this curve proportionally to its $g_0$-arclength preserving both its initial point and its orientation.
Define the $S_1$-action
\[
\rho_t:S^1 \to \Diff(S_0 M)
\]
such that $\rho_t(\theta)$ takes~$v$ to the tangent vector of this new curve at the point of parameter~$\theta$.
Since $\Psi_t$ is a diffeomorphism, the map $\rho_t(\theta)$ is also a diffeomorphism of~$S_0 M$.
Clearly, the $S^1$-action~$\rho_t$ on~$S_0 M$ is free and satisfies the symmetry property 
\begin{equation} \label{eq:sym}
\rho_t(\theta)(-v) = - \rho_t(-\theta)(v)
\end{equation}
for every $t \in [0,\infty)$, $\theta \in S^1=\R/2\pi \Z$ and $v \in S_0 M$.
Moreover, every $\rho_t$-orbit projects to an embedding of~$S^1$ into~$M$ by the canonical projection~$S_0 M \to M$.
It is also worth pointing out that the expression of~$\rho_t(\theta)(v)$ vary smoothly with respect to $t \in [0,\infty)$, $\theta \in S^1$ and $v \in S_0 M$.

\medskip

Thus defined, the actions~$\rho_t$ satisfy the following convergence result which implies Theorem~\ref{theo:flowrp2} by passing to the quotient.

\forget
\section{Convergence of the geodesic flow under the curvature flow} \label{sec:action}

In this section, we define a deformation of the geodesic flow of balanced Zoll Finsler spheres to the geodesic flow of the canonical round sphere, which relies on the isotopy of diffeomorphisms previously introduced.

\medskip

Let $(M,F)$ be a balanced Zoll Finsler two-sphere.
Recall that the geodesic flow of~$F$ on~$S M$ induces a free $S^1$-action on~$S_0 M$ by conjugacy by the musical diffeomorphisms $\pi_\flat$ and~$\pi_\sharp$.
This $S^1$-action denoted by
\[
\rho:S^1 \to \Diff(S_0 M)
\]
is defined as
\[
\rho(\theta)(v) = \pi_\flat[\gamma_{\pi_\sharp(v)}'(\theta)]
\]
for every $\theta \in S^1$ and $v \in S_0 M$.
One can deform this action through the one-parameter family of $S^1$-actions~$(\Psi_t)_* \,  \rho$, \cf~Section~\ref{sec:flow}, which can be expressed as
\[
\begin{array}{ccc}
S^1 \times S_0 M & \to & S_0 M \\
(\theta,v) & \mapsto & \pi_\flat[(\gamma_{\pi_\sharp(v)}^t)'(\theta)]
\end{array}
\]
They are free since embedded loops remain embedded through the curvature flow, \cf~Theorem~\ref{theo:prop}.\eqref{c2}.
However, it is not clear whether they converge to a free $S^1$-action as $t$ goes to infinity.
Instead, we consider a different family of $S^1$-actions defined as follows. 

\medskip

Consider  an embedded loop $\gamma:S^1 \to M$.
Let~$\hat{\gamma}$ be the orientation-preserving reparametrization of~$\gamma$ proportional to its $g_0$-arclength with the same initial point, \ie, $\hat{\gamma}(0) = \gamma(0)$.
Denote by $s:S^1 \to S^1$ the corresponding change of parameter with 
\[
\hat{\gamma}(\theta) = \gamma(s(\theta)).
\]
Given $t \in [0,1]$, define $\gamma_t:S^1 \to M$ as
\[
\gamma_t(\theta) = \gamma(t \, s(\theta) + (1-t) \, \theta)
\]
for every $\theta \in S^1$.
Clearly, the loop~$\gamma_t$ is a regular orientation-preserving reparametrization of~$\gamma$ with the same initial point.
By construction, the isotopy~$(\gamma_t)$ obtained by reparametrization connects~$\gamma_0=\gamma$ to~$\gamma_1=\hat{\gamma}$.
Note that the isotopy~$(\gamma_t)$ smoothly depends on~$\gamma$ (where the loop space on~$M$ is endowed with the smooth Fr\'echet structure) as do the initial curve~$\hat{\gamma}$ and the change of parameter~$s$.

\medskip

Applying these reparametrizations to the $F$-geodesics~$\gamma_v$, we obtain an $S^1$-action deformation~$\rho_t$ of~$\rho$ for $t \in [0,1]$
\[
\rho_t:S^1 \to \Diff(S_0 M)
\]
with $\rho_0=\rho$ such that 
\begin{equation} \label{eq:rhot}
\rho_t(\theta)(v) = \pi_\flat[(\gamma_{\pi_\sharp(v)})_t'(\theta)]
\end{equation}
for every $\theta \in S^1$ and $v \in S_0 M$.
Here, the map $\pi_\flat: TM \setminus \{ 0\} \to S_0 M$ is the radial projection onto~$S_0 M$.
Although the $S^1$-action~$\rho_1$ differs from the original action~$\rho$ induced by the geodesic flow of~$F$, the $\rho_t$-orbits remain the same for every $t \in [0,1]$.
Actually, the $S^1$-action~$\rho_1$ is given by the following expression:
\[
\rho_1(\theta)(v) = \pi_\flat[(\widehat{\gamma_{\pi_\sharp(v)}})'(\theta)].
\]

Now, we extend the $S^1$-action deformation~$(\rho_t)$ from $t \in [0,1]$ to $t \geq 1$ by letting
\begin{equation} \label{eq:rhott}
\rho_{t+1}(\theta)(v) = \pi_\flat[(\widehat{\gamma_{\pi_\sharp(\Psi_t^{-1}(v))}^t})'(\theta)].
\end{equation}
for every $t \geq 0$, $\theta \in S^1$ and $v \in S_0 M$.
In other words, for every $v \in S_0 M$, consider the unique curve~$\gamma_u^t$ tangent to~$v$ at~$\theta=0$ and pointing in the same direction as~$v$.
Reparametrize this curve proportionally to its $g_0$-arclength preserving both its initial point and its orientation.
Then, for every $\theta \in S^1$, the diffeomorphism~$\rho_t(\theta)$ takes~$v$ to the tangent vector of this new curve at the point of parameter~$\theta$.
Clearly, the $S^1$-action~$\rho_t$ on~$S_0 M$ is free and satisfies the symmetry property 
\begin{equation} \label{eq:sym}
\rho_t(\theta)(-v) = - \rho_t(-\theta)(v)
\end{equation}
for every $t \in [0,\infty)$, $\theta \in S^1=\R/2\pi \Z$ and $v \in S_0 M$.
Moreover, every $\rho_t$-orbit projects to an embedding of~$S^1$ into~$M$ by the canonical projection~$S_0 M \to M$.
It is also worth pointing out that the expressions~\eqref{eq:rhot} and~\eqref{eq:rhott} vary (piecewise) smoothly with respect to $t \in [0,\infty)$, $\theta \in S^1$ and $v \in S_0 M$.

\medskip

Thus defined, the actions~$\rho_t$, whose construction relies on Theorem~\ref{theo:diffeo}, satisfy the following convergence result which implies Theorem~\ref{theo:flowrp2} by passing to the quotient.
\forgotten

\begin{theorem} \label{theo:flowS2}
Let $(M,F)$ be a balanced Zoll Finsler two-sphere.
Then the smooth free $S^1$-actions 
\[
\rho_t:S^1 \times S_0 M \to S_0 M
\]
converge to the action~$\rho_\infty$ induced by the geodesic flow of the canonical round sphere.

Furthermore, for $t \in [0,\infty]$, every $\rho_t$-orbit projects to an embedding of~$S^1$ into~$M$ under the canonical projection $S_0 M \to M$.
\end{theorem}

\begin{proof}
It follows from Theorem~\ref{theo:prop} that for every $\varepsilon > 0$ and every $t\geq 0$ large enough, the unparametrized loops~$\gamma_u^t$ have curvature at most~$\varepsilon$ on the canonical round sphere.
In particular, these loops are uniformly close to the equators to which they are tangent (in the smooth Fr\'echet topology).
By construction, this implies that the action~$\rho_t$ is close to the action~$\rho_\infty$ induced by the $g_0$-geodesic flow for $t$ large enough.

The last statement about the orbits of~$\rho_t$ is also satisfied since these orbits are transverse to the fibers of $S_0M \to M$ and project to the images of the $F$-geodesics under the curvature flow (which do not self-intersect).
\end{proof}

\begin{remark}
This convergence result shows that the one-parameter family of $S^1$-actions~$(\rho_t)$ is defined for $t \in [0,\infty]$.
In the rest of this article, we will consider a reparametrization~$(\rho_\tau)$ of this family for $\tau \in [0,1]$ with $\tau = \frac{t}{t+1}$.
\end{remark}

\section{Crofton formula on Zoll Finsler two-spheres}

We review some constructions on Zoll Finsler two-spheres~$(M,F)$ all of whose geodesics are of length~$2\pi$, including the general Crofton formula on Finsler surfaces. 

\medskip



Let $\mathcal{L}: TM \to T^*M$ be the Legendre transform of the Lagrangian~$\frac{1}{2} F^2$.
Since $F$ is quadratically convex, the Legendre transform is a diffeomorphism between $TM$ and~$T^*M$.
By homogeneity of~$F$, it preserves the norm on each fiber of the bundle vectors $TM$ and~$T^*M$.
In particular, it induces a diffeomorphism between the unit sphere bundle and the unit co-sphere bundle $SM$ and~$S^*M$.
Geometrically, this diffeomorphism is defined as follows: for every vector $v \in S_xM$, the image~$\mathcal{L}(v)$ of~$v$ is the unique covector of~$S_x^*M$ such that $\mathcal{L}(v)(v)=1$.

\medskip

From now on, we will identify $TM$ with~$T^*M$ and $SM$ with~$S^*M$ via the Legendre transform.
Recall that we also identify~$SM$ with the unit tangent bundle~$S_0 M$ of the canonical round sphere under the musical diffeomorphisms~$\pi_\flat$ and~$\pi_\sharp$, \cf~\eqref{eq:sharp} and~\eqref{eq:flat}.
With these identifications, the action~$\rho_F$ of~$S^1$ on~$S M$ given by the geodesic flow of~$F$ induces an action on~$S^* M$ by conjugacy by the Legendre transform, namely the co-geodesic flow of~$F$, and an action on~$S_0^* M$ by conjugacy both by the Legendre transform and the musical diffeomorphisms.
Despite the risk of confusion, all these $S^1$-actions will be denoted by~$\rho_F$.


\forget

\medskip

There is another diffeomorphism $\vartheta:S_0^* M \to S_0 M$ depending on the Legendre transform and defined as follows: for every $\xi \in S_0^* M_{|x}$. the image~$\vartheta(\xi)$ of~$\xi$ is the unique vector~$v \in S_0 M_{|x}$ lying in the kernel of~$\xi$, \ie $\xi(v)=0$, such that the basis $(v,\mathcal{L}^{-1}(\xi))$ of~$T_x M$ is positively oriented.
The conjugacy of the geodesic flow~$\rho_F$ on~$S_0 M$ by the diffeomorphism~$\vartheta$ defines a free $S^1$-action on~$S_0^* M$
\[
\rho^\perp_F : S^1 \to \Diff(S_0^* M)
\]
given by 
\begin{equation} \label{eq:perp}
\rho^\perp_F(\theta)(\xi) = \vartheta^{-1}[\rho_F(\theta)(\vartheta(\xi))]
\end{equation}
for every $\theta \in S^1$ and $\xi \in S_0^* M$.

\forgotten

\medskip

Let $\alpha$ be the tautological one-form on~$T^* M$.
By definition,
\[
\alpha_\xi(V) = \xi(d \pi_\xi(V))
\]
for every $\xi \in T^* M$ and $V \in T_\xi T^* M$, where $\pi:T^* M \to M$ is the canonical surjection.
From the Liouville theorem, the tautological one-form~$\alpha$ (and so the symplectic form~$\omega = d\alpha$) is invariant under the co-geodesic flow of any Finsler metric.
Observe also that the $S^1$-orbits of~$\rho_F$ on~$S^* M$ and~$S_0^* M$ are transverse to the contact structures given by the kernels of~$\alpha$ and~$\pi_\sharp^*(\alpha)$.
Now, the pull-back of the $3$-form $\alpha \wedge d\alpha$ under the inclusion map \mbox{$S^* M \hookrightarrow T^* M$} defines a volume form on~$S^* M$.
(That is, the pull-back of~$\alpha$ is a contact one-form on~$S^* M$.)
Since $S^* M$ is the unit cotangent bundle of a Finsler sphere all of whose geodesics are closed of lengh~$2\pi$, the integral of this volume form on~$S^* M$ does not depend on the Finsler metric and is equal to~$\pm \, 8 \pi^2$ by a result of A.~Weinstein, \cf~\cite[\S2.C]{besse}.
That is,
\begin{equation} \label{eq:wein}
\int_{S^* M} \alpha \wedge d\alpha = \int_{S_0^* M} \alpha \wedge d\alpha = \pm \, 8 \pi^2.
\end{equation}

\forget
Recall that the quotient manifold theorem asserts that if $G$ is a Lie group acting smoothly, freely and properly on a smooth manifold~$N$, then the quotient space $N/G$ is a topological manifold with a unique smooth structure such that the quotient map $N \to N/G$ is a smooth submersion.
This result applies to the $S^1$-action on~$S^*_0 M$ given by the co-geodesic flow~$\rho_F$.

Denote by 
\[
\Gamma_F=S^*_0 M/\rho_F
\]
the quotient manifold and by 
\[
q_F:S^*_0 M \to \Gamma_F
\]
the quotient submersion.
The manifold~$\Gamma_F$ represents the space of unparametrized oriented geodesics of the Zoll two-sphere~$(M,F)$.
It is diffeomorphic to~$S^2$.
\forgotten

By the quotient manifold theorem, the $S^1$-action on~$S^*_0 M$ given by the co-geodesic flow~$\rho_F$ gives rise to a quotient manifold
\[
\Gamma_F=S^*_0 M/\rho_F
\]
diffeomorphic to~$S^2$, representing the space of unparametrized oriented geodesics of the Zoll Finsler two-sphere~$(M,F)$, and a quotient submersion
\[
q_F:S^*_0 M \to \Gamma_F.
\]
By construction, the map~$q_F$ takes a unit cotangent vector of~$M$ to the unparametrized oriented $F$-geodesic of~$M$ with the Legendre transform of this unit cotangent vector as initial condition.
Thus, the projection~$\pi(q_F^{-1}(\gamma))$ of a fiber over~$\gamma$ represents the unparametrized closed geodesic of~$(M,F)$ given by~$\gamma \in \Gamma_F$.
We will sometimes identify~$\gamma$ with~$\pi(q_F^{-1}(\gamma))$.

\medskip

Consider the double fibration
\[
\xymatrix{
  & S^*_0 M \ar[dl]_\pi \ar[dr]^{q_F} \ar[r]^i & T^*M \\
 M & & \Gamma_F}
\]
where $i: S^*_0 M \xrightarrow{\pi_\flat^*} S^* M \hookrightarrow T^*M$ is given by the canonical injection once $S^*_0 M$ is identified with~$S^* M$, and $\pi:S_0^* M \to M$ is the canonical surjection.
Note that 
the product map $\pi \times q_F:S_0^* M \to M \times \Gamma_F$ is an embedding.
From~\cite{besse}, there exists a unique symplectic form~$\lambda_F$ on~$\Gamma_F$ such that 
\begin{equation} \label{eq:lambda}
q_F^* \, \lambda_F = i^* \omega.
\end{equation}

\medskip

The general Crofton formula on Finsler surfaces can be stated as follows.

\begin{theorem}[{\cite[Theorem~5.2]{AB06}}] \label{theo:crofton}
With the previous notations, the length of every smooth curve~$c$ on~$(M,F)$ satisfies
\begin{equation} \label{eq:crofton}
\length_F(c) = \frac{1}{4} \int_{\gamma \in \Gamma_F} \#(\gamma \cap c) \, |\lambda_F|
\end{equation}
where $|\lambda_F|$ is the smooth positive area density on~$\Gamma_F$ induced by the symplectic form~$\lambda_F$. 
\end{theorem}

\begin{remark}
Strictly speaking, the integrand in the formula~\eqref{eq:crofton} should be $\#(\pi(q_F^{-1}(\gamma)) \cap c)$ instead of~$\#(\gamma \cap c)$, but as aforementioned, we identify the elements~$\gamma$ in~$\Gamma_F$ with the unparametrized geodesics~$\pi(q_F^{-1}(\gamma))$ they represent.
\end{remark}

\begin{remark}
The Crofton formula~\eqref{eq:crofton} shows that the Zoll Finsler metric~$F$ is uniquely determined by the submersion $q_F:S_0^*M \to \Gamma_F$ (and the symplectic form~$\lambda_F$ on~$\Gamma_F$ derived from~$q_F$).
\end{remark}

\section{Deforming Zoll Finsler two-spheres}

In this section, we construct a natural deformation of Zoll Finsler two-spheres to the canonical round two-sphere by applying the Crofton formula to the orbits of the converging family of circle actions given by the curvature flow, \cf~Theorem~\ref{theo:flowS2}.

\medskip

Consider a Zoll Finsler two-sphere~$(M,F)$ all of whose geodesics are of length~$2\pi$.
Let $\rho$ be a smooth free $S^1$-action on~$S_0^* M$ whose orbits are transverse to the contact structure given by the kernel of~$\pi_\sharp^*(\alpha)$.
(Recall that we identify $S^*M$ with~$S_0^*M$ under the musical diffeomorphisms~$\pi_\flat^*$ and~$\pi_\sharp^*$.)
Denote by $q_\rho:S_0^* M \to \Gamma_\rho$ the submersion induced by the free $S^1$-action~$\rho$, where $\Gamma_\rho = S_0^* M / \rho$.

\medskip

Define a volume form~$\Omega$ on~$S_0^* M$ as follows
\begin{equation} \label{eq:Omega}
\Omega = \pi_\sharp^*(\alpha \wedge d\alpha).
\end{equation}
Along with the $S^1$-action~$\rho$, this volume form gives rise to a two-form~$\bar{\omega}=\bar{\omega}_\rho$ on~$S_0^* M$ by the following averaging construction:

\begin{equation} \label{eq:omegabar}
\bar{\omega}_\rho = \frac{1}{2\pi} \int_{S^1} \rho(\theta)^*[i_\nu (\Omega)] \, d\theta
\end{equation}
where $\nu=\nu_\rho$ is the vector field on~$S_0^* M$ generated by the $S^1$-action~$\rho$, that is, 
\[
\nu(\xi) = \frac{d}{d\theta}_{|\theta=0} \rho(\theta)(\xi)
\]
for every $\xi \in S_0^* M$.
Thus,
\begin{equation} \label{eq:omegabar2}
\bar{\omega}_\xi(u,v) = \frac{1}{2\pi} \,  \int_{S^1} \Omega_{\rho(\theta)(\xi)}(d\rho(\theta)_{|\xi}(u),d\rho(\theta)_{|\xi}(v),\tfrac{d}{d\theta} \rho(\theta)(\xi) ) \, d\theta
\end{equation}
for every $\xi \in S_0^* M$ and $u,v \in T_\xi S_0^* M$.
Here, despite the ambiguity in the notation, $d\rho(\theta)_{|\xi}$ denotes the differential of the diffeomorphism $\rho(\theta):S_0^* M \to S_0^* M$ at~$\xi$.
By construction, the two-form~$\bar{\omega}_\rho$ is \mbox{$\rho$-invariant} and projects to a two-form~$\lambda_\rho$ on the quotient surface~$\Gamma_\rho=S^*_0 M/\rho$ with 
\[
\bar{\omega}_\rho=q_\rho^* \, \lambda_\rho.
\]
Up to the multiplicative factor~$\frac{1}{2\pi}$, the form~$\lambda_\rho$ is the two-form induced by~$\Omega$ by integration along the fibers of~$q_\rho$ (that is, the push-forward of~$\Omega$ by the fibration~$q_\rho$).
Note that both two-forms $\bar{\omega}_\rho$ and~$\lambda_\rho$ are determined by~$\rho$.

\medskip 

We have the following straightforward result.

\begin{lemma}
The two-form~$\lambda_\rho$ does not vanish (and so defines an area-form on~$\Gamma_\rho$).
\end{lemma}

\begin{proof}
Consider two independent vectors $\bar{u}$ and~$\bar{v}$ based at the same point tangent to~$\Gamma_\rho$.
Let $u$ and~$v$ be two lifts of~$\bar{u}$ and~$\bar{v}$, based at the same point~$\xi \in S_0^* M$, under the submersion~$q_\rho$.
That is, the vectors~$u$ and~$v$ tangent to~$S_0^* M$ at~$\xi$ project to~$\bar{u}$ and~$\bar{v}$ under~$dq_\rho$.
By construction, we have
\[
\lambda_\rho(\bar{u},\bar{v}) = \bar{\omega}_\xi(u,v).
\]
Furthermore, for every $\theta \in S^1$, the vectors $d\rho(\theta)_{|\xi}(u)$ and $d\rho(\theta)_{|\xi}(v)$ also project to~$\bar{u}$ and~$\bar{v}$ under~$dq_\rho$.
Now, since the vector~$\frac{d}{d\theta} \rho(\theta)(\xi)$ is tangent to the fibers of~$q_\rho$, it follows that the three vectors $d\rho(\theta)_{|\xi}(u)$, $d\rho(\theta)_{|\xi}(v)$ and $\frac{d}{d\theta} \rho(\theta)(\xi)$ form a basis of~$T_{\rho(\theta)(\xi)} S_0^* M$.
Thus, the value of the volume form~$\Omega_{\rho(\theta)(\xi)}$ at these vectors is nonzero (and so of constant sign) for every $\theta \in S^1$.
From the expression~\eqref{eq:omegabar2}, we conclude that both $\bar{\omega}_\xi(u,v)$ and $\lambda_\rho(\bar{u},\bar{v})$ are nonzero.
Hence the result.
\end{proof}

By fiber integration and Fubini's theorem, we have the following relation
\[
\int_{S_0^* M} \Omega = 2 \pi \int_{\Gamma_\rho} \lambda_\rho.
\]
In particular, the integral of the (non-vanishing) area-from~$|\lambda_\rho|$ over~$\Gamma_\rho$ is equal to~$4 \pi$ when~$\rho$ is given by the co-geodesic flow of~$F$, \cf~\eqref{eq:wein}.

\medskip

From now on, suppose that the $S^1$-action~$\rho$ is symmetric, that is,
\[
\rho(\theta)(-\xi) = - \rho(-\theta)(\xi)
\]
for every $\theta \in S^1=\R/2\pi \Z$ and~$\xi \in S_0^* M$.
This is the case when $\rho$ is given by the co-geodesic flow of~$F$, and more generally, by its one-parameter family of deformations~$(\rho_\tau)_{0 \leq \tau \leq 1}$ defined at the end of Section~\ref{sec:flow}, \cf~\eqref{eq:sym}.
The symmetric $S^1$-action~$\rho$ induces a submersion
\[
p_\rho:S_0^* M \to \Gamma_\rho
\]
as follows
\[
p_\rho(\xi)=\gamma \Leftrightarrow \ker \xi \mbox{ is tangent to the non-oriented curve } \gamma
\]
where $\xi \in S_0^* M$ and $\gamma \in \Gamma_\rho$.
Clearly, the map~$p_\rho$ induces a submersion $PT^*M \to \Gamma_\rho$ to the quotient $PT^*M = S_0^* M/\pm$ representing the space of contact elements of~$M$.
The fibers of this submersion are transverse to those of the canonical projection $PT^*M \to M$.
Furthermore, these fibers are Legendrian with respect to the contact structure induced by~$\alpha$.
Indeed, let $\xi \in p_\rho^{-1}(\gamma)$ and $V \in T_\xi p_\rho^{-1}(\gamma)$.
The vector~$d \pi_\xi (V)$, based at~$\pi(\xi)$, is tangent to~$\gamma$.
Hence, $\xi(d \pi_\xi (V))=0$ since $p_\rho(\xi)=0$, by definition of~$p_\rho$.
That is, $\alpha_\xi(V)=0$ for every $V \in T_\xi p_\rho^{-1}(\gamma)$.
Note also that the product map $\pi \times p_\rho: S_0^* M \to M \times \Gamma_\rho$ is an embedding.

\medskip

We can now apply the results of~\cite{AB} asserting that a (non-vanishing) area-form on~$\Gamma_\rho$ gives rise to a Finsler metric on~$M$ via the Crofton formula.

\begin{theorem}[{\cite[Theorem~2.2]{AB}}] \label{theo:AB}
With the previous notations, there exists a unique Finsler metric~$F_\rho$ on~$M$ satisfying the Crofton formula
\begin{equation} \label{eq:croftonrho}
\length_{F_\rho}(c) = \frac{1}{4} \int_{\gamma \in \Gamma_\rho} \#(\gamma \cap c) \, |\lambda_\rho|
\end{equation}
for any smooth curve~$c$ on~$M$.

Morevover, the Finsler metric~$F_\rho$ admits the following expression: for every $x \in M$, there exists a unique non-vanishing one-form~$\beta_x$ on~$S_x^*M$ such that
\begin{equation} \label{eq:Frho}
F_\rho(x;v) = \int_{\xi \in S_x^*M} |\xi(v)| \, \beta_x
\end{equation}
for every $v \in T_x M$, where the non-vanishing one-form~$\beta_x$ is defined for every $\xi \in S_x^* M$ by the relation
\begin{equation} \label{eq:beta}
(\bar{\omega}_\rho)_{(x,\xi)} = \pi^* \xi \wedge \beta_{(x,\xi)}.
\end{equation}
Furthermore, the one-form~$\beta_x$ smoothly depends on~$x$.
\end{theorem}

Note that the relation~\eqref{eq:beta} allows us to define the one-form~$\beta_{(x,\xi)}$ in a unique way only on~$T_\xi S_x^* M$, not on~$T_{(x,\xi)} S^* M$.

\medskip

Before making use of Theorem~\ref{theo:AB}, let us mention three important observations that will be useful in the proof of Theorem~\ref{theo:retractS2} below.

\medskip

The first one deals with the dependance of~$F_\rho$ with respect to~$\rho$.

\begin{remark}
By construction, both forms~$\bar{\omega}_\rho$ and~$\beta=\beta_\rho$, \cf~\eqref{eq:omegabar} and~\eqref{eq:beta}, continuously vary with the $S^1$-action~$\rho$.
It follows from the expression of~$F_\rho$, \cf~$\eqref{eq:Frho}$, that the Finsler metric~$F_\rho$ continuously varies with~$\rho$ too.
\end{remark}

The second observation is about the geodesics of~$F_\rho$.

\begin{remark} \label{rk:geodesic}
In the two-dimensional case, it follows from the Crofton formula that the geodesics of~$F_\rho$ are exactly the curves~$\pi(q_\rho^{-1}(\gamma))$ given by~$\gamma \in \Gamma_\rho$, see~\cite[Theorem~3.3]{AB}.
\end{remark}

The third observation deals with $S^1$-actions arising from the co-geodesic flow of the Zoll Finsler two-sphere~$(M,F)$.

\begin{remark} \label{rk:startend}
In the special case when $\rho$ is given by the co-geodesic flow of~$F$, that is, $\rho=\rho_F$, the following properties hold.
The vector field~$\nu$ agrees with the co-geodesic vector field~$X_F$ of~$F$ on~$S_0^*M \simeq S^* M$.
(In the sequel, the musical isomorphisms will be omitted.)
Since $i_{X_F}(\alpha) = 1$, we deduce that $i_\nu(\Omega) = d\alpha$ from the expression of the volume form~$\Omega$, \cf~\eqref{eq:Omega}.
Now, by the Liouville theorem, the symplectic form~$d\alpha$ is $\rho_F$-invariant, that is, $\rho_F(\theta)^* (d\alpha) = d\alpha$ for every $\theta \in S^1$.
This shows that $\bar{\omega} = d\alpha$.
Hence, $\lambda_\rho = \lambda_F$ by definition of~$\lambda_F$, \cf~\eqref{eq:lambda}.
Now, as $\Gamma_\rho=\Gamma_F$ and $q_\rho = q_F$, it follows from the Crofton formulas \eqref{eq:crofton} and~\eqref{eq:croftonrho} that $F_\rho = F$.
\end{remark}

We can now proceed to the proof of the following result.

\begin{theorem} \label{theo:retractS2}
The space of balanced Zoll Finsler metrics on the two-sphere whose geodesic length is equal to~$2\pi$ strongly deformation retracts to the canonical round metric on the two-sphere.

Furthermore, this strong deformation retraction is induced by the curvature flow on the canonical round two-sphere.
\end{theorem}

\begin{proof}
Let $(M,F)$ be a balanced Zoll Finsler two-sphere all of whose geodesics are of length~$2\pi$.
Consider the deformation~$(\rho_\tau)_{0 \leq \tau \leq 1}$ of $S^1$-actions on~$S_0^* M$ defined at the end of Section~\ref{sec:flow}, where $\rho_0$ is given by the co-geodesic flow of~$F$ and $\rho_1$ agrees with the co-geodesic flow of the canonical round metric~$g_0$.
Since every $\rho_\tau$-orbit projects to an embedding of~$S^1$ into~$M$ by the canonical projection~$S_0^* M \to M$, these orbits are transverse to the contact structure given by the kernel of~$\pi_\sharp^*(\alpha)$.

Define a one-parameter family of Finsler metrics~$F_\tau = F_{\rho_\tau}$ for \mbox{$0 \leq \tau \leq 1$} as in Theorem~\ref{theo:AB}.
We will also write $q_\tau$,  $\Gamma_\tau$ and $\lambda_\tau$ for~$q_{\rho_\tau}$, $\Gamma_{\rho_\tau}$ and~$\lambda_{\rho_\tau}$.
From Remark~\ref{rk:startend}, the metric deformation~$(F_\tau)$ starts at~$F$, that is, $F_0=F$.
By Remark~\ref{rk:geodesic}, the geodesics of~$F_\tau$ are precisely the simple closed curves represented by~$\Gamma_\tau$, namely, the curves~$\pi(q_\tau^{-1}(\gamma))$ where $\gamma$ runs over~$\Gamma_\tau$.
Thus, the geodesics of~$F_\tau$ agree with the images of the geodesics of~$F$ under the curvature flow at some time depending on~$\tau$.
In particular, the metrics~$F_\tau$ are balanced Zoll Finsler metrics, \cf~Proposition~\ref{prop:equiv}, and the metric~$F_1$ has the same geodesics as the canonical round metric~$g_0$.

At this point, we do not claim that $F_1$ agrees with~$g_0$.
Indeed, by construction, the metric~$F_1$ is determined by its space of geodesics $\Gamma_1 = \Gamma_{g_0}$ \emph{and} a smooth positive measure~$|\lambda_1|$ on it, which, in this case, may differ from~$|\lambda_{g_0}|$.
This leads us to extend the metric deformation~$(F_\tau)$ in a natural way as follows.
For every $1 \leq \tau \leq 2$, define a Finsler metric~$F_\tau$ as in Theorem~\ref{theo:AB} with $\Gamma_\tau=\Gamma_1$, $q_\tau=q_1$ and 
\[
|\lambda_\tau|=(2-\lambda) |\lambda_1| + (\tau-1) |\lambda_{g_0}|.
\]
As previously, the geodesics of these new metrics agree with those of the canonical round sphere, but now, $F_2$ is equal to~$g_0$, since $|\lambda_2| = |\lambda_{g_0}|$.

We can also estimate the lengths of the geodesics of~$F_\tau$ as follows. 
Let~$c_0$ be a geodesic of~$F_\tau$.
By Theorem~\ref{theo:2}, every pair of distinct closed geodesics of~$F_\tau$ has exactly two intersection points.
Hence, $\#(\gamma \cap c_0)=2$ for almost every~$\gamma \in \Gamma_\tau$.
Since the integral of~$|\lambda_\tau|$ equals~$4 \pi$, it follows from the Crofton formula~\eqref{eq:croftonrho} that the length of~$c_0$ is equal to~$2 \pi$.

In conclusion, the metric deformation~$(F_\tau)$ gives rise to a retraction from the space of balanced Zoll Finsler metrics on the two-sphere with geodesic length~$2 \pi$ to the canonical round metric~$g_0$ on the two-sphere.
\end{proof}

\begin{remark}
Theorem~\ref{theo:retractrp2} follows from Theorem~\ref{theo:retractS2} by taking the orientable double cover of the Zoll Finsler projective plane since, by construction, the strong deformation retraction on the two-sphere passes to the quotient by the antipodal map.
\end{remark}



\section{Appendix}

In this appendix, we show how flexible Zoll Finsler metrics are.
Given a closed Zoll Finsler manifold~$M$, we construct an infinite-dimensional family of Zoll Finsler metrics on~$M$ with the same geodesic length.
The construction -- proceeding by local perturbations of the initial metric -- is loosely constrained and fairly easy to implement.
Furthermore, all local Zoll perturbations are obtained by following this construction.

\medskip

The results in this section are not new, but we provide details we were unable to find in the literature.
Our presentation follows the approach of~\cite{Iv13} and \cite{BI} regarding boundary rigidity problems (and used in~\cite{chen} to study Finsler tori without conjugate points).
It largely borrows from~\cite{chen}. 

\medskip

Let $(M,F)$ be a closed Finsler $n$-manifold.
Fix an open ball~$D$ in~$M$ of radius less than $\frac{1}{4} \inj(M)$ centered at~$x_0$ .
For every $p \in \partial D$, the arclength parametrized geodesic~$\gamma_p$ with $\gamma_p(0)=x_0$ and~$\gamma_p(r)=p$ defines a point~$\gamma_p(-r)$ lying in~$\partial D$ denoted by~$-p$.
The set of points of~$D$ equidistant from~$p$ and~$-p$ forms a hypersurface~$H_p$ passing through~$x_0$ which divides~$D$ into two connected components: $H_p^+$ containing~$p$ and~$H_p^-$ containing~$-p$.

Define a smooth function $f:\partial D \times D \to \R$ as
\[
f(p,x) = 
\begin{cases}
d(H_p,x) & \mbox{if } x \in H_p^+ \\
-d(H_p,x) & \mbox{otherwise}
\end{cases}
\]
For every $p \in \partial D$, denote $f_p=f(p,\cdot)$.
The function~$f$ is an \emph{enveloping function}, that is, it satisfies the following conditions:
\begin{enumerate}
\item for every $x \in D$, the map $p \mapsto df_p(x)$ is a diffeomorphism from~$\partial D$ to the boundary of a quadratically convex body of~$T_x^* M$ containing the origin; \label{env1}
\item for every $p \in \partial D$, we have $f_{-p}=-f_p$. \label{env2}
\end{enumerate}
In our case, the boundary of the quadratically convex body in~\eqref{env1} is~$S_x^* M$ since $F^*(df_p(x))=1$ for every $p \in \partial D$ and~$x \in D$.
The condition~\eqref{env2} ensures that the convex bodies in~\eqref{env1} are symmetric with respect to the origin.

The distance function induced by~$F$ can be written as
\begin{equation} \label{eq:dF}
d_F(x,y) = \sup_{p \in \partial D} f_p(x) - f_p(y)
\end{equation}
for every $x,y \in D$.
Similarly, the Finsler metric~$F$ can be expressed as
\[
F(v) = \sup_{p \in \partial D} df_p(v)
\]
for every $v \in TD$.

\medskip

Consider a sufficiently small $C^\infty$-perturbation~$\tilde{f}$ of~$f$ such that $\tilde{f}$ is an enveloping function which agrees with~$f$ on~$\partial D \times U$, where $U$ is a tubular neighborhood of~$\partial D$ in~$D$.
Note that if the perturbation is small enough, the condition~\eqref{env1} is immediately satisfied by~$\tilde{f}$.

Define a new Finsler metric~$\tilde{F}$ on~$D$ as
\[
\tilde{F}(v) = \sup_{p \in \partial D} d\tilde{f}_p(v)
\]
for every $v \in TD$.
The induced distance is given by
\begin{equation} \label{eq:dFF}
d_{\tilde{F}}(x,y) = \sup_{p \in \partial D} \tilde{f}_p(x) - \tilde{f}_p(y)
\end{equation}
for every $x,y \in D$.
Here, the reversibility of~$\tilde{F}$ (and symmetry of~$d_{\tilde{F}}$) follows from the condition~\eqref{env2}.
Also, the function~$\tilde{f}_p$ satisfies $\tilde{F}^*(d\tilde{f}_p(x))=1$ for every~$x \in D$.
Moreover, for every~$x \in D$, there exists a unique $\tilde{F}$-unit tangent vector~$v \in T_x D$ such that $d\tilde{f}_p(v)=1$.
This tangent vector smoothly depends on~$x$ (and~$p$) and defines an $\tilde{F}$-unit vector field~$\tilde{\nabla} \tilde{f}_p$ on~$D$, called the \emph{$\tilde{F}$-gradient} of~$\tilde{f}_p$. 
Since $F$ and~$\tilde{F}$ agree on~$U$ (as do $f$ and~$\tilde{f}$ on~$\partial D \times U$), we can extend~$\tilde{F}$ by letting $\tilde{F}=F$ outside~$D$.

\medskip

The geodesics of~$\tilde{F}$ can be described as follows.

\begin{proposition}
The geodesics of~$\tilde{F}$ agree with the integral curves of~$\tilde{\nabla} \tilde{f}_p$ on~$D$, with~$p \in \partial D$.
Furthermore, these curves are $\tilde{F}$-minimizing on~$D$.
\end{proposition}

\begin{proof}
Consider an integral curve $\gamma$ of~$\tilde{\nabla} \tilde{f}_p$.
By construction, the curve~$\gamma$ is parametrized by its $\tilde{F}$-arclength and $d\tilde{f}_p(\gamma'(t)) = 1$ for every $t \in [a,b]$.
Thus,
\[
b-a = \int_a^b d\tilde{f}_p(\gamma'(t)) \, dt = \tilde{f}_p(\gamma(b)) - \tilde{f}_p(\gamma(a)) \leq b-a
\]
since $\tilde{f}_p$ is $\tilde{F}$-nonexpanding.
This implies 
\[
d_{\tilde{F}}(\gamma(b),\gamma(a)) = \tilde{f}_p(\gamma(b)) - \tilde{f}_p(\gamma(a)) = b-a
\]
Hence, the arc~$\gamma$ is a minimizing $\tilde{F}$-geodesic on~$D$.

Conversely, let $\gamma_v$ be the $\tilde{F}$-geodesic induced by some $\tilde{F}$-unit tangent vector~$v \in T_x D$ based at~$x$.
From the condition~\eqref{env1}, there exists a (unique) $p \in \partial D$ such that $v$ agrees with~$\tilde{\nabla} \tilde{f}_p$ at~$x$.
The integral curve of~$\tilde{\nabla} \tilde{f}_p$ passing through~$x$ is an $\tilde{F}$-geodesic with the same initial condition~$v$ as the $\tilde{F}$-geodesic~$\gamma_v$.
Therefore, the two $\tilde{F}$-geodesics agree on~$D$.
\end{proof}

\medskip

Let~$\tilde{\gamma}$ be an $\tilde{F}$-geodesic arc of~$D$.
Since~$\tilde{\gamma}$ is $\tilde{F}$-minimizing, it leaves~$D$ through two points~$x$ and~$y$ in~$\partial D$.

\begin{proposition} \label{prop:geod}
Let $\gamma$ be the $F$-geodesic arc of~$D$ with the same endpoints as~$\tilde{\gamma}$.
Then, the arcs $\gamma$ and~$\tilde{\gamma}$ satisfy $\ell_{\tilde{F}}(\tilde{\gamma}) = \ell_F(\gamma)$ and have the same tangent vectors at their endpoints $x$ and~$y$.
\end{proposition}

\begin{proof}
Let us show this latter statement holds for the tangent vectors of~$\gamma$ and~$\tilde{\gamma}$ at~$y$ (and so at~$x$, by symmetry).
Fix a point~$y_+$ close to~$y$ lying slightly outside~$\bar{D}$ on the $F$-geodesic extension of~$\gamma$.
Denote by~$\alpha$ the $F$-minimizing arc (lying on the $F$-geodesic extension of~$\gamma$) joining~$y$ to~$y_+$, \cf~Figure~\ref{fig:geod}.
The arc~$\alpha$ is also minimizing for~$\tilde{F}$ since $\tilde{F}=F$ in the neighborhood of~$y$. 

\medskip

\begin{figure}[htb] 
\def\svgwidth{4.5cm} 
\hspace{1cm} 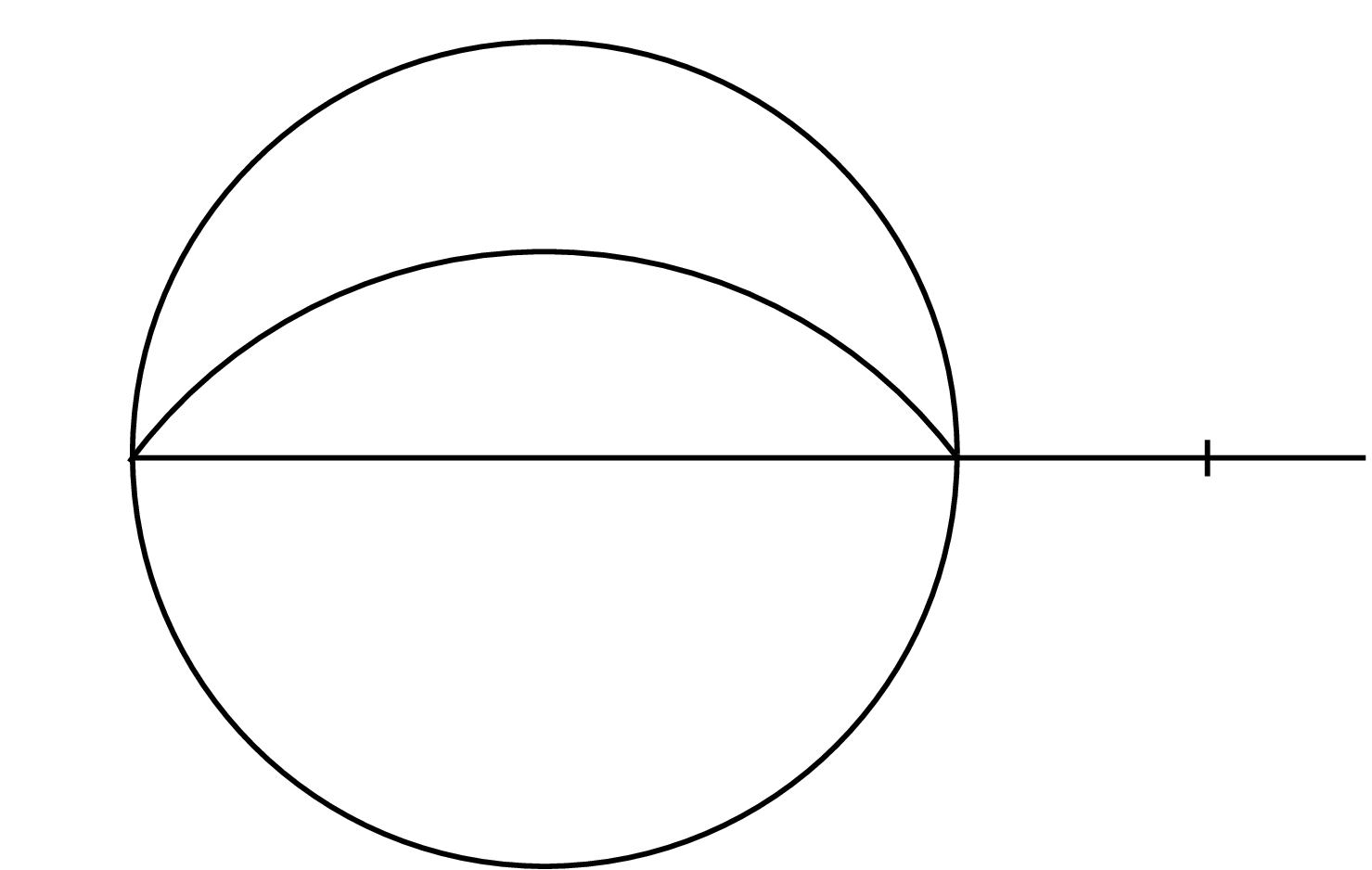
\caption{Geodesics of~$D$ for $F$ and~$\tilde{F}$} \label{fig:geod}
\end{figure}

Since $f$ and~$\tilde{f}$ agree on~$\partial D \times U$, we deduce from the expression of~$d_F$ and~$d_{\tilde{F}}$, \cf~\eqref{eq:dF} and~\eqref{eq:dFF}, that $d_F=d_{\tilde{F}}$ not only for pairs of points in a neighborhood of~$y$ but on~$\partial D \times \partial D$.
Thus, for every $z \in \partial D$ close enough to~$y$, we have
\[
d_F(x,z) + d_F(z,y_+) = d_{\tilde{F}}(x,z) + d_{\tilde{F}}(z,y_+).
\]
The infimum of the left-hand side of this equation over such~$z$ is attained for~$z=y$.
By a first variation argument applied to the right-hand side of the equation, we deduce that $\tilde{\gamma} \cup \alpha$ is smooth.
(Recall that $\tilde{\gamma}$ is an $\tilde{F}$-minimizing arc joining $x$ to~$y$, and that $\alpha$ is an $\tilde{F}$-minimizing arc joining $y$ to~$y_+$.)
Thus, the unit tangent vectors of~$\tilde{\gamma}$ at~$y$ is the same as the unit tangent vector of~$\alpha$ (and so~$\gamma$) at~$y$.

Now, since $\gamma$ and~$\tilde{\gamma}$ are minimizing with respect to~$F$ and~$\tilde{F}$, the relation~$d_F = d_{\tilde{F}}$ on~$\partial D \times \partial D$ also implies
\[
\ell_{\tilde{F}}(\tilde{\gamma}) = d_{\tilde{F}}(x,y) = d_F(x,y) = \ell_F(\gamma).
\]
\end{proof}

Suppose now that $F$ is a Zoll Finsler metric and that, in addition, the radius of~$D$ is small enough so that the geodesics of~$F$ pass at most once through~$D$.
We derive from Proposition~\ref{prop:geod} that the geodesics of~$\tilde{F}$ are closed and agree with those of~$F$ outside~$D$.
Furthermore, they are simple and have the same length as those of~$F$.
In other words, $\tilde{F}$ is also a Zoll Finsler metric with the same geodesic length as~$F$. \\

To make sure the Finsler metrics are not pairwise isometric, we can proceed as follows.
Consider the Mahler product volume of a two-dimensional centrally symmetric convex body~$B$ in an affine plane~$E$
\[
\PP(B) = \vol(B \times B^*)
\]
where $\vol$ is the volume induced by the canonical symplectic form on~$E \times E^*$.
The Mahler product volume is an affine invariant of~$B$ which satisfies the sharp lower bound $\PP(B) \geq 8$ in dimension two, with equality if and only if $B$ is a parallelogram, \cf~\cite[Theorem~2.3.4]{thompson}.

Let $B$ be an $n$-dimensional centrally symmetric convex body in an affine $n$-space.
Define the affine planar content of~$B$ as
\[
c(B) = \min_P \PP(B \cap P)
\]
where $P$ runs over the planes passing through the center of~$B$.
The affine planar content~$c(B)$ is an affine invariant of~$B$ which satisfies $c(B) \geq 8$.
Note that this sharp lower bound is never attained by a quadratically convex body.

To produce an infinite-dimensional family of \emph{non-isometric} Zoll Finsler metrics as above, we perturb the metric~$F$ in a neighborhood of a point~$x_0$ such that $c(B_{x_0}^F) \leq c(B_x^F)$ for every $x \in M$, where $B_x^F$ is the $F$-unit tangent ball in~$T_x M$.
Let $B_{x_0}^F \cap P$ be a planar section of~$B$ with minimal Mahler product volume.
Fix $p_0 \in \partial D$ such that the $F$-gradient of~$f_{p_0}$ at~$x_0$, namely~$\nabla f_{p_0}(x_0)$, lies in~$P$.
We can perturb $f:\partial D \times D \to \R$ in the neighborhood of~$(p_0,x_0)$ (and~$(-p_0,x_0)$ to respect the symmetry condition~\eqref{env2}) through a family of enveloping functions~$\tilde{f}_\tau$ which agree with~$f$ in the neighborhood of~$\partial D \times \partial D$ and such that the content $c(\tilde{F}_\tau) := \min_{x \in M} c(B_x^{\tilde{F}_\tau})$ of the Finsler metric~$\tilde{F}_\tau$ induced by~$\tilde{f}_\tau$ decreases with~$\tau$.
Since the content of a Finsler metric is preserved by isometries, this construction provides a family~$\tilde{F}_\tau$ of non-isometric Zoll Finsler metrics on~$M$ with the same geodesic length.

\begin{remark}
Note that we would derive the same result by considering other affine invariants for centrally symmetric convex bodies.
We decided to work with the affine planar content because its minimum is never attained by a quadratically convex body and because it is not difficult to decrease its value through small perturbations.
\end{remark}

\end{document}